\documentclass[12pt]{article}
 
 \usepackage{amsmath,amssymb,amscd,amsthm,esint}
 
\usepackage{graphics,amsmath,amssymb,amsthm,mathrsfs}

\newtheorem{theorem}{Theorem}[section]

\newtheorem{lemma}[theorem]{Lemma}

\newtheorem{remark}[theorem]{Remark}

\def\xxint#1#2#3{{\setbox0=\hbox{$#1{#2#3}{\int}$}
  \vcenter{\hbox{$#2#3$}}\kern-.5\wd0}}

\def\varep{\varepsilon}

\newcommand{\average}{-\!\!\!\!\!\!\int}

\begin{document}

\title
{\bf Lipschitz Estimates in \\  Almost-Periodic Homogenization}

\author{Scott N. Armstrong and Zhongwei Shen\thanks{Supported in part by NSF grant DMS-1161154.}}
 
\date{ }

 \maketitle

\begin{abstract}

We establish uniform Lipschitz estimates 
for second-order elliptic systems in divergence form with rapidly oscillating, almost-periodic coefficients. We give interior estimates as well as estimates up to the boundary in bounded $C^{1,\alpha}$ domains with either Dirichlet or Neumann data. The main results extend those in the periodic setting due to Avellaneda and Lin \cite{AL-1987,AL-1991} for interior and Dirichlet boundary estimates and later Kenig, Lin, and Shen \cite{KLS1} for the Neumann boundary conditions. In contrast to these papers, our arguments are constructive (and thus the constants are in principle computable) and the results for the Neumann conditions are new even in the periodic setting, since we can treat non-symmetric coefficients. We also obtain uniform $W^{1,p}$ estimates.

\end{abstract}
\medskip




\tableofcontents
\section{Introduction}

The primary purpose of this paper is to establish uniform Lipschitz estimates
for a family of elliptic operators with rapidly oscillating, almost-periodic coefficients,
arising in the theory of homogenization.
More precisely,  we consider the linear elliptic operator
\begin{equation}\label{operator}
\mathcal{L}_\varep
=-\text{\rm div} \big( A(x/\varep)\nabla \big)
=-\frac{\partial }{\partial x_i} \left\{ a_{ij}^{\alpha\beta} (x/\varep) \frac{\partial}{\partial x_j} \right\}, \qquad
\qquad \varep>0
\end{equation}
(the summation convention is used throughout).
Let $A(y)=\big( a_{ij}^{\alpha\beta} (y)\big)$ be real and bounded in $\mathbb{R}^d$, where 
$1\le i,j\le d$ and $1\le \alpha, \beta\le m$.
Throughout the paper we will assume that
\begin{equation}\label{ellipticity}
\mu |\xi|^2 \le a_{ij}^{\alpha\beta} (y)\xi_i^\alpha\xi_j^\beta \le \mu^{-1} |\xi |^2 \quad 
\text{ for any }  y\in \mathbb{R}^d \text{ and } \xi=(\xi_i^\alpha) \in \mathbb{R}^{m\times d},
\end{equation}
where $\mu>0$, and
\begin{equation}\label{uniform-ap}
\lim_{R\to \infty} \sup_{y\in \mathbb{R}^d} \inf_{\substack{ z\in \mathbb{R}^d\\ |z|\le R}}
\| A(\cdot +y)-A(\cdot +z)\|_{L^\infty(\mathbb{R}^d)} = 0.
\end{equation}
Notice that if $A$ is bounded and continuous in $\mathbb{R}^d$,
then $A$ satisfies (\ref{uniform-ap}) if and only if 
 $A$ is \emph{uniformly almost-periodic} in $\mathbb{R}^d$, i.e.,
 each entry of $A$ is the uniform limit of a sequence of trigonometric polynomials.
  We define the following modulus, which quantifies the almost periodic assumption:
\begin{equation}
\label{rho}
\rho (R):=\sup_{y\in \mathbb{R}^d} \inf_{\substack{ z\in \mathbb{R}^d\\ |z|\le R}}
\| A(\cdot +y)-A(\cdot +z)\|_{L^\infty(\mathbb{R}^d)}.
\end{equation}

Given a bounded $C^{1, \alpha}$ domain $\Omega\subset \mathbb{R}^d$, we are interested in estimating the quantity
 $\|\nabla u_\varep\|_{L^\infty(\Omega)}$, uniformly in $\varep>0$,
for weak solutions~$u^\varep$ of Dirichlet problem
\begin{equation}\label{DP-1}
\mathcal{L}_\varep (u_\varep) =F \quad \text{ in } \Omega \quad \text{ and } \quad u_\varep =f \quad 
\text{ on } \partial\Omega,
\end{equation}
as well as those of the Neumann problem
\begin{equation}\label{NP-1}
\mathcal{L}_\varep (u_\varep)=F \quad \text{ in } \Omega
\quad
\text{ and } \quad \frac{\partial u_\varep}{\partial\nu_\varep}
=g \quad \text{ on } \partial\Omega.
\end{equation}
In (\ref{NP-1}) we have used $\partial u_\varep/\partial \nu_\varep$
to denote the conormal derivative $n(x) A(x/\varep)\nabla u_\varep (x)$ on $\partial\Omega$,
where $n(x)$ is the outward unit normal to $\partial\Omega$.

To ensure that we have Lipschitz estimates at small scales, we assume that $A$ is uniformly 
H\"older continuous, i.e., there exist $\tau>0$ and $\lambda\in (0,1]$
such that
\begin{equation}\label{H-continuity}
|A(x)-A(y)|\le \tau |x-y|^\lambda \quad \text{ for any } x,y \in \mathbb{R}^d.
\end{equation}

The following are the main results of the paper.

\begin{theorem}\label{main-theorem-Lip}
Suppose that $A(y)$ satisfies  uniform ellipticity (\ref{ellipticity}) and H\"older continuity conditions (\ref{H-continuity}).
Suppose also that there exist $N>5/2$ and $C_0>0$ such that
\begin{equation}\label{decay-condition}
\rho (R) \le C_0 \big[ \log R\big]^{-N} \qquad \text{ for any } R\ge 2.
\end{equation}
Let $\Omega$ be a bounded $C^{1, \alpha}$ domain in $\mathbb{R}^d$
for some $\alpha>0$.
Let $u_\varep\in H^1(\Omega; \mathbb{R}^m)$ be a weak solution of Dirichlet problem
(\ref{DP-1}). Then
\begin{equation}\label{Lip-estimate-0}
\| \nabla u_\varep\|_{L^\infty(\Omega)}
\le C \left\{ \| F\|_{L^p(\Omega)} +\| f\|_{C^{1, \beta} (\partial\Omega)} \right\},
\end{equation}
where $p>d$, $\beta \in (0, \alpha)$, and $C$ depends only on $p$, $\beta$, 
$A$, and $\Omega$.
\end{theorem}

\begin{theorem}\label{main-theorem-Lip-N}
Suppose that $A(y)$ satisfies  (\ref{ellipticity}) and (\ref{H-continuity}).
Also assume that the decay condition (\ref{decay-condition}) holds for some  $N>3 $ and $C_0>0$.
 Let $\Omega$ be a bounded $C^{1, \alpha}$ domain in $\mathbb{R}^d$
for some $\alpha>0$.
Let $u_\varep\in H^1(\Omega; \mathbb{R}^m)$ be a weak solution of the Neumann problem
(\ref{NP-1}). Then
\begin{equation}\label{Lip-estimate-N-0}
\| \nabla u_\varep\|_{L^\infty(\Omega)}
\le C \left\{ \| F\|_{L^p(\Omega)} +\| g\|_{C^{\beta} (\partial\Omega)} \right\},
\end{equation}
where $p>d$, $\beta \in (0, \alpha)$, and $C$ depends only on $p$, $\beta$, 
$A$, and $\Omega$.
\end{theorem}

Note that if $A(y)$ is periodic, then $\rho (R)=0$ for $R$ sufficiently large and thus
satisfies the assumption (\ref{decay-condition}) for any $N>1$.
In this case the Lipschitz estimate (\ref{Lip-estimate-0}) for the Dirichlet problem (\ref{DP-1}) 
in $C^{1, \alpha}$ domains
was established by Avellaneda and Lin \cite{AL-1987}
under the conditions (\ref{ellipticity}) and (\ref{H-continuity}).
This classical result was recently extended by Kenig, Lin, and Shen in \cite{KLS1}, where estimate
(\ref{Lip-estimate-N-0}) was established for solutions of the Neumann problem (\ref{NP-1}) in the periodic setting,
under an additional symmetry condition $A^*(y)=A(y)$, i.e., 
 $a_{ij}^{\alpha\beta}(y)=a_{ji}^{\beta\alpha} (y)$ for any
 $1\le i, j\le d$ and $1\le \alpha, \beta\le m$.
Our Theorems \ref{main-theorem-Lip} and \ref{main-theorem-Lip-N}
further extend the main results in \cite{AL-1987} and \cite{KLS1} to the almost-periodic
setting. We point out that Theorem \ref{main-theorem-Lip-N} is new even in the periodic
setting, as the symmetry condition $A^*=A$ is not required.
We also remark that the Lipschitz estimates in Theorems \ref{main-theorem-Lip} and 
\ref{main-theorem-Lip-N} are sharp in the sense that there is no uniform modulus of continuity for the gradient of solutions unless $\text{\rm div}(A)=0$. As for the $C^{1,\alpha}$ assumption on the domain $\Omega$, we note that
Lipschitz estimates may fail on a $C^1$ domain even for harmonic functions.

The proof of uniform estimates in both \cite{AL-1987} and \cite{KLS1} is based on a compactness argument
that originated from the study of regularity theory in the calculus of variation and minimal surfaces.
The argument, which was introduced in \cite{AL-1987} to the study of homogenization,
extends readily to the almost-periodic setting in the case of uniform H\"older estimates.
In fact it was proved in \cite{Shen-2014} that if $u_\varep$ is a weak solution of the Dirichlet problem:
\begin{equation}\label{DP-2}
\mathcal{L}_\varep (u_\varep) =F +\text{\rm div} (h) \quad
\text{ in } \Omega \quad \text{ and } \quad
u_\varep =f \quad \text{ on } \partial\Omega,
\end{equation}
where $\Omega$ is a bounded $C^{1,\alpha}$ domain in $\mathbb{R}^d$,
then
\begin{equation}\label{H-estimate}
\aligned
\| u_\varep\|_{C^\beta (\overline{\Omega})}
\le C \bigg\{ \| f\|_{C^\beta (\partial\Omega)}
 &+\sup_{\substack{x\in \Omega\\ 0<r<r_0}} r^{2-\beta} \average_{B(x,r)\cap \Omega} |F|\\
 &+\sup_{\substack{x\in \Omega\\ 0<r<r_0}}
 r^{1-\beta}
 \left(\average_{B(x,r)\cap \Omega} |h|^2\right)^{1/2}\bigg\}
 \endaligned
 \end{equation}
 for any $\beta \in (0,1)$,
 where $r_0=\text{\rm diam} (\Omega)$ and $C$ depends only on $\beta $, $A$, and $\Omega$.
 However, for Lipschitz estimates,
 the approach in \cite{AL-1987, KLS1} relies  on the Lipschitz estimates
 for interior and boundary correctors in a crucial way.
 It is not clear how to extend this to the almost-periodic setting,
 as any estimate of correctors in a non-periodic setting is far from trivial,
 even in the interior case.

Our proof of Theorem \ref{main-theorem-Lip} and \ref{main-theorem-Lip-N}
will be based on a rather general scheme for proving Lipschitz estimates at large scale in homogenization.
The scheme, which was motivated by the compactness argument in \cite{AL-1987},
was recently formulated  and used by the first author and C. Smart in \cite{Armstrong-Smart-2014}
for convex integral functionals with random coefficients. The idea, rather than arguing by contradiction (by compactness), is to apply a $C^{1,\alpha}$ Campanato iteration directly. For this we need to show that the ``flatness" of a solution $u$ (how well it is approximated by an affine function) improves on smaller scales, e.g., for some $\theta\in(0,1/4)$,
\begin{equation}\label{flatness}
\aligned
\frac{1}{r\theta } 
& \inf_{\substack{ M\in \mathbb{R}^{m\times d}\\ 
q\in \mathbb{R}^m}} \left(\average_{B_{\theta r}} |u(x)-M x-q|^2\, dx \right)^{1/2}\\
&\qquad\qquad \le \frac{1}{2} \left( \frac{1}{r} 
\inf_{\substack{ M\in \mathbb{R}^{m\times d}\\ 
q\in \mathbb{R}^m}}
  \left(\average_{B_{r}} |u(x)- M x-q|^2 \, dx \right)^{1/2} \right).
 \endaligned
\end{equation}
Since solutions of the \emph{homogenized} equation satisfy such an estimate (on all scales), we indeed have~(\ref{flatness}) up to the error arising in homogenization. For large balls, we may expect this error to be much smaller than the improvement in the flatness. Therefore, if we can control  the error in homogenization effectively, we may hope to iterate the improvement of flatness estimate down to microscopic scales. Indeed, as we show in Theorem~\ref{g-theorem-1} (which is a slight modification of~\cite[Lemma 5.1]{Armstrong-Smart-2014}), this scheme yields a uniform Lipschitz estimate down to the microscopic scale, provided that the rate of homogenization is sufficiently fast: an algebraic (or even Dini-type) convergence rate suffices.

Such error estimates were recently proved by the second author \cite{Shen-2014} for solutions $u_\varep$ of Dirichlet problem (\ref{DP-1}). In particular, it was shown that
\begin{equation}\label{error-estimate-0}
\| u_\varep -u_0\|_{L^2(\Omega)} \le C\, \omega (\varep) \| u_0\|_{W^{2, p} (\Omega)},
\end{equation}
where $p>d$ and $\omega (\varep)$ is a modulus on $(0,1]$, with $\omega (0+)=0$,  
which can be given explicitly using the modulus $\rho (R)$ in (\ref{rho}) (see Section 2). The decay conditions on $\rho(R)$ in Theorems~\ref{main-theorem-Lip}
and~\ref{main-theorem-Lip-N} are used precisely to ensure that we have Dini-type rates for homogenization.
We mention that condition (\ref{decay-condition}) holds, for example, if $A(y)$ is quasi-periodic in
$\mathbb{R}^d$ with frequencies satisfying the so-called Kozlov (C) condition \cite{Kozlov-1979}.
In fact it was proved in \cite{Shen-2014} that the Kozlov (C) condition implies that
$\rho(R)\le C_0 (R+1)^{-\lambda}$ for some $\lambda>0$.

We do not know if the decay assumptions on $\rho(R)$ in Theorems~\ref{main-theorem-Lip}
and~\ref{main-theorem-Lip-N} can be weakened  substantially or whether uniform Lipschitz estimates hold for general uniformly almost periodic coefficients. However, we remark that the  scheme for proving Lipschitz estimates formalized in Theorem~\ref{g-theorem-1} is a quite general tool that can be useful in other circumstances. It applies, for example, to the Poisson equation
\begin{equation}\label{e.lapf}
-\Delta u = \text{\rm div} f
\end{equation}
and yields a Lipschitz estimate on $u$ precisely in the case that $f$ is $C^\alpha$ (or Dini continuous). Likewise, a straightforward modification of Theorem~\ref{g-theorem-1} yields a statement that implies the classical Schauder estimates. Of course, if $f$ is merely continuous, then it is well-known that  solutions of~\eqref{e.lapf} may fail to be Lipschitz continuous. This suggests that, in analogy, the decay conditions on $\rho(R)$ are natural and perhaps even necessary. 

The general scheme mentioned above makes minimal use of the structure of the equation and in particular does not involve correctors in a direct manner
(though indirectly via approximation requirements).
As a result, it can be adapted surprisingly well
 for proving Lipschitz estimates up to the boundary with either Dirichlet or Neumann boundary conditions.
The key step then is to establish suitable error estimates of $\| u_\varep -u_0\|_{L^2(\Omega)}$,
not necessarily sharp, for local weak solutions with Dirichlet or Neumann condition.
This will be achieved by considering the function
$$
w_\varep =u_\varep (x)- v_0 (x) -\varep \chi_T (x/\varep) \nabla v_0 (x),
$$
where $T=\varep^{-1}$, $\mathcal{L}_0 (v_0) =0$, 
and $\chi_T(y)$ denotes the approximate correctors for $\mathcal{L}_\varep$.
The proof  relies on the pointwise estimates of $\chi_T$ obtained in \cite{Shen-2014}.
 In the case of Neumann conditions our argument also
 requires uniform Lipschitz estimates of $\chi_T$, which follow from uniform interior Lipschitz 
 estimates. However, as we indicated earlier, our approach does not use boundary
 correctors.

Let $G_\varep (x,y)$ denote the matrix of Green functions for $\mathcal{L}_\varep$ in $\Omega$, with
pole at $y$. It follows from the proof of Theorem \ref{main-theorem-Lip} that for any $x,y\in \Omega$ and
$x\neq y$,
\begin{equation}\label{G-estimate-1}
|\nabla_x G_\varep (x,y) |+|\nabla_y G_\varep (x,y)|\le C\, |x-y|^{1-d}
\end{equation}
and
\begin{equation}\label{G-estimate-2}
|\nabla_x \nabla_y G_\varep (x,y)|\le C \, |x-y|^{-d},
\end{equation}
where $C$ depends only on $A$ and $\Omega$.
This, in particular, implies that the Poisson kernel $P_\varep (x,y)$ for
$\mathcal{L}_\varep$ in $\Omega$ satisfies 
\begin{equation}\label{P-estimate}
|P_\varep (x,y)|\le \frac{C\, \text{\rm dist} (x, \partial\Omega)}{|x-y|^d}
\end{equation}
for any $x\in \Omega$ and $y\in \partial\Omega$.
As in the periodic setting \cite{AL-1987-ho, AL-1987},
estimate (\ref{P-estimate}) yields the following.

\begin{theorem}\label{main-theorem-max}
Suppose that $A$ and $\Omega$ satisfy the same conditions as in Theorem \ref{main-theorem-Lip}.
Let $1<p<\infty$.
Let $u_\varep$ be the solution of the $L^p$ Dirichlet problem
\begin{equation}\label{DP-3}
\mathcal{L}_\varep (u_\varep)=0 \quad \text{ in } \Omega \quad \text{ and } \quad
u_\varep =f \quad \text{ on } \partial\Omega
\end{equation}
with $(u_\varep)^*\in L^p(\partial\Omega)$,
where $f\in L^p(\partial\Omega; \mathbb{R}^m)$ and $(u_\varep)^*$ denotes the non-tangential 
maximal function of $u_\varep$.
Then
\begin{equation}\label{max-estimate}
\| (u_\varep)^*\|_{L^p(\partial\Omega)} \le C_p \, \| f\|_{L^p(\partial\Omega)},
\end{equation}
where $C_p$ depends only on $p$, $A$, and $\Omega$.
Furthermore, if $f\in L^\infty(\partial\Omega)$, then
\begin{equation}\label{m-p}
\| u_\varep \|_{L^\infty(\Omega)} \le C\, \| f\|_{L^\infty(\partial\Omega)},
\end{equation}
where $C$ depends only on  $A$ and $\Omega$.
\end{theorem}

In this paper we also study the uniform $W^{1,p}$ estimates for $\mathcal{L}_\varep$.
Related results in the periodic setting may be found in \cite{ AL-1987, AL-1991,Shen-2008, song-2012, KLS1, Geng-Shen-2014}.
We emphasize that the H\"older condition (\ref{H-continuity}) is not assumed in the following two
theorems.

\begin{theorem}\label{main-theorem-W-1-p}
Suppose that $A(y)$ is uniformly almost-periodic in $\mathbb{R}^d$ and satisfies (\ref{ellipticity}).
Also assume that $A(y)$ satisfies the condition (\ref{decay-condition}) for some $N>(3/2)$.
Let $\Omega$ be a bounded $C^{1, \alpha}$ domain in $\mathbb{R}^d$ for some $\alpha>0$ and $1<p<\infty$.
Let $u_\varep \in W^{1,p}(\Omega; \mathbb{R}^m)$ be a weak solution of the Dirichlet problem 
(\ref{DP-2}),
where $h=(h_i^\alpha)\in L^p(\Omega; \mathbb{R}^{m\times d})$, $F\in L^p(\Omega; \mathbb{R}^m)$ and 
$f\in B^{p, 1-\frac{1}{p}} (\partial\Omega; \mathbb{R}^m)$.
Then
\begin{equation}\label{W-1-p}
\| u_\varep\|_{W^{1, p}(\Omega)}
\le C_p \left\{ \| h\|_{L^p(\Omega)} +\| F\|_{L^p(\Omega)} + \| f\|_{B^{p, 1-\frac{1}{p}} (\partial\Omega)}\right\},
\end{equation}
where $C_p$ depends only on $p$, $A$, and $\Omega$.
\end{theorem}

\begin{theorem}\label{main-theorem-W-1-p-N}
Suppose that $A$ and $\Omega$ satisfy the same conditions as in Theorem \ref{main-theorem-W-1-p}.
Let  $1<p<\infty$.
Let $u_\varep\in W^{1,p}(\Omega; \mathbb{R}^m)$ be a weak solution to
\begin{equation}\label{NP-2}
\mathcal{L}_\varep (u_\varep)= \text{\rm div} (h) +F \quad \text{ in } \Omega \quad \text{ and }\quad 
\frac{\partial u_\varep}{\partial \nu_\varep} =g - n\cdot h \quad \text{ on } \partial\Omega.
\end{equation}
Then
\begin{equation}\label{W-1-p-N}
\| \nabla u_\varep\|_{L^p(\Omega)}
\le C_p \left\{ \| h\|_{L^p(\Omega)}
+\| F\|_{L^p(\Omega)} + \| g\|_{B^{-\frac{1}{p},p}(\partial\Omega)} \right\},
\end{equation}
where $C_p$ depends only on $p$, $A$, and $\Omega$.
\end{theorem}

We conclude this section with some notation and comments on bounding constants $C$.
We will use $\average_E f=\frac{1}{|E|}\int_E f$ to denote the $L^1$ average of $f$ over a set $E$.
For a ball $B=B(x, r)$ we use $\alpha B$ to denote $B(x, \alpha r)$.
We will use $C$ to denote constants that may depend on $d$, $m$, $A(y)$, $\Omega$,
and other relevant parameters, but never on $\varep$.
It is important to note that since our assumptions
on $A$ are invariant under translation and rotation,
the constants $C$ will be invariant under any translation and rotation of $\Omega$.
This allows us to use freely translation and rotation to simply the argument.
As for rescaling, we observe that if $\mathcal{L}_\varep  (u_\varep)=F$ and $v(x)=u_\varep (rx)$,
then $\mathcal{L}_{\varep/r} (v)=G$, where $G(x)=r^2 F(rx)$.



\section{Homogenization and convergence rates}
\setcounter{equation}{0}

Let $\mathcal{L}_\varep =-\text{div} \big(A(x/\varep)\nabla \big)$. Throughout this section we
assume that $A(y)= \big( a_{ij}^{\alpha\beta} (y) \big)$ is uniformly almost-periodic in $\mathbb{R}^d$
and satisfies the ellipticity condition (\ref{ellipticity}).
The H\"older continuity (\ref{H-continuity}) and decay condition (\ref{decay-condition})
will not be used here.

\subsection{The homogenized operator and qualitative homogenization}

To define the homogenized operator $\mathcal{L}_0$, we first introduce the space 
$B^2(\mathbb{R}^d)$, the $L^2$ space of almost-periodic functions in the sense of Bezikovich.

A function $f$ in $L^2_{\rm loc} (\mathbb{R}^d)$ is said to belong to $B^2(\mathbb{R}^d)$
if $f$ is a limit of a sequence of trigonometric polynomials in $\mathbb{R}^d$ with respect to the
semi-norm
$$
\| f\|_{B^2} =\limsup_{R\to \infty} \left\{ \average_{B(0,R)} |f|^2 \right\}^{1/2}.
$$
For $f\in L^1_{\rm loc} (\mathbb{R}^d)$,
a number $\langle f\rangle$ is called the mean value of $f$ if
$$
\lim_{\varep\to 0^+} \int_{\mathbb{R}^d}
f(x/\varep) \varphi (x)\, dx =\langle f \rangle \int_{\mathbb{R}^d} \varphi
$$
for any $\varphi\in C_0^\infty(\mathbb{R}^d)$.
It can be shown that if $f, g\in B^2(\mathbb{R}^d)$, then $fg$ has the mean value.
Under the equivalent relation that $f\sim g$ if $\| f-g\|_{B^2}=0$,
the vector space $B^2(\mathbb{R}^d)/\sim$ becomes a Hilbert space with the inner product defined 
by $(f,g)=\langle f,g\rangle$.

Let $V^2_{\rm pot}$ (reps. $V^2_{\rm sol}$) denote the closure in $B^2(\mathbb{R}^d; \mathbb{R}^{m\times d})$
of potential (reps. solenoidal) trigonometric polynomials with mean value zero.
Then
$$
B^2(\mathbb{R}^d; \mathbb{R}^{m\times d})
=V^2_{\rm pot} \oplus V^2_{\rm sol} \oplus \mathbb{R}^{m\times d}.
$$
By the Lax-Milgram Theorem and the ellipticity condition (\ref{ellipticity}), for any
$1\le j\le d$ and $1\le \beta\le m$, there exists a unique $\psi_j^\beta=\big(\psi_{ij}^{\alpha\beta}\big)
\in V^2_{\rm pot}$ such that
\begin{equation}\label{homo-equation}
\langle a_{ik}^{\alpha\gamma} \psi_{kj}^{\gamma\beta} \phi_{i}^\alpha\rangle
=-\langle a_{ij}^{\alpha\beta} \phi_i^\alpha\rangle 
\quad \text{ for any } \phi=\left( \phi_i^\alpha\right)\in V^2_{\rm pot}.
\end{equation}
Let $\widehat{A}=\big ( \widehat{a}_{ij}^{\alpha\beta}\big)$, where
\begin{equation}\label{homo-coeff}
\widehat{a}_{ij}^{\alpha\beta}
=\langle a_{ij}^{\alpha\beta}\rangle + \langle a_{ik}^{\alpha\gamma} \psi_{kj}^{\gamma\beta}\rangle.
\end{equation}
The homogenized operator for $\mathcal{L}_\varep$ is given by $\mathcal{L}_0
=-\text{\rm div} \big(\widehat{A}\nabla\big)$.
We refer the reader to \cite{Jikov-1994} for details (also see earlier work in \cite{Kozlov-1979, Kozlov-1980,
Papanicolaou-1979}).

The proof of the following theorem may be found in \cite{Jikov-1994}.
\begin{theorem}\label{homo-theorem}
Let $\Omega$ be a bounded Lipschitz domain in $\mathbb{R}^d$. For
$F\in H^{-1}(\Omega; \mathbb{R}^m)$ and $\varep>0$,
let $u_\varep \in H^1(\Omega; \mathbb{R}^m)$ be a weak solution of
$\mathcal{L}_\varep (u_\varep) =F$ in $\Omega$.
Suppose that for some subsequence $\{ u_{\varep^\prime} \}$,
$u_{\varep^\prime} \to u_0$ weakly in $H^1(\Omega; \mathbb{R}^m)$ and
$A(x/\varep^\prime)\nabla u_{\varep^\prime} \to G$ weakly in $L^2(\Omega; \mathbb{R}^{m\times d})$.
Then $G=\widehat{A} \nabla u_0$ in $\Omega$.
\end{theorem}

The homogenization of Dirichlet problem (\ref{DP-1}) and the Neumann problem (\ref{NP-1})
follows readily from Theorem \ref{homo-theorem}.
For $\varep\ge 0$, $F\in H^{-1}(\Omega; \mathbb{R}^m)$ and
$f\in H^{1/2}(\partial\Omega; \mathbb{R}^m)$,
let $u_\varep\in H^1(\Omega; \mathbb{R}^m)$ be the unique weak solution of (\ref{DP-1}).
Then $u_\varep \to u_0$ weakly in $H^1(\Omega; \mathbb{R}^m)$ 
and strongly in $L^2(\Omega; \mathbb{R}^m)$, as $\varep \to 0$.
Similarly, if $\int_\Omega u_\varep =\int_\Omega u_0=0$,
 the solution of the Neumann problem (\ref{NP-1}) with $F\in H^{-1}(\Omega; \mathbb{R}^m)$
and $g \in H^{-1/2}(\partial\Omega; \mathbb{R}^m)$ 
converges weakly in $H^1(\Omega; \mathbb{R}^m)$ to the solution of the Neumann problem:
$\mathcal{L}_0 (u_0) =F$ in $\Omega$ and $\partial u_0/{\partial \nu_0}=g$
on $\partial\Omega$, where $\partial u_0/{\partial \nu_0}
= n\widehat{A}\nabla u_0$.

\subsection{Quantitative estimates for the approximate correctors}

To study the convergence rates of $u_\varep$ to $u_0$,
we need to introduce the approximate correctors $\chi_T = \big(\chi_{T, j}^\beta\big)$.
Let $P_j^\beta (y)=y_j e^\beta$, where $1\le j\le d$, $1\le \beta\le m$, and
$e^\beta=(0, \dots, 1, \dots, 0)$ with $1$ in the $\beta^{th}$ position.
For each $T>0$, the function $u=\chi_{T, j}^\beta$ is defined as the weak solution of
\begin{equation}\label{ac}
-\text{\rm div} \big( A(y)\nabla u \big) + T^{-2} u =\text{\rm div} \big( A(y)\nabla P_j^\beta\big)
\qquad \text{ in } \mathbb{R}^d,
\end{equation}
with the property
$$
\sup_{x\in \mathbb{R}^d} \| u\|_{H^1(B(x, 1))} <\infty.
$$
It is not hard to show that
\begin{equation}
\sup_{x\in \mathbb{R}^d} \average_{B(x,T)} \big( |\nabla \chi_T|^2 +T^{-2} |\chi_T|^2\big)\le C,
\end{equation}
where $C$ depends only on $d$, $m$, and $\mu$ (the almost-periodicity of $A$ is not needed).

For $\sigma \in (0, 1]$ and $T\ge 1$, define
\begin{equation}\label{Theta}
\Theta_\sigma (T)=\inf_{0<R\le T}
\left\{ \rho (R) +\left(\frac{R}{T}\right)^\sigma \right\}.
\end{equation}
The following theorem was proved in \cite{Shen-2014}.

\begin{theorem}\label{ac-theorem}
Suppose that $A$ is uniformly almost-periodic in $\mathbb{R}^d$
and satisfies (\ref{ellipticity}).
Let $\sigma \in (0,1)$ and $T\ge 1$.
Then, for any $x,y\in \mathbb{R}^d$,
\begin{equation}\label{ac-Holder}
|\chi_T (x)-\chi_T (y)|\le C_\sigma\, T^{1-\sigma}\,  |x-y|^\sigma,
\end{equation}
and
\begin{equation}\label{ac-bound}
T^{-1} \| \chi_T \|_{L^\infty(\mathbb{R}^d)} \le C_\sigma \, \Theta_\sigma (T),
\end{equation}
where $C_\sigma$ depends only on $\sigma$ and $A$.
\end{theorem}

The rest of this section is devoted to the study of error estimates 
of $\| u_\varep -u_0\|_{L^2(\Omega)}$.
The material is divided into two subsections.
The first subsection treats Dirichlet boundary condition, while the second handles
the Neumann boundary condition.


\subsection{Convergence rates: Dirichlet boundary condition}

We begin by using Theorem \ref{ac-theorem} to extend a result in \cite{Shen-2014}.

\begin{lemma}\label{lemma-H-1}
Let $\Omega$ be a bounded Lipschitz domain in $\mathbb{R}^d$.
Let $u_\varep\in H^1(\Omega; \mathbb{R}^m)$ be the weak solution of (\ref{DP-1}).
Let
\begin{equation}
w_\varep=u_\varep -v_0 -\varep \chi_T (x/\varep)\nabla v_0 -v_\varep,
\end{equation}
where $T=\varep^{-1}$, $v_0\in W^{2,2}(\Omega;\mathbb{R}^m)$,
$\mathcal{L}_0 (v_0)=F$ in $\Omega$, and $v_\varep\in H^1(\Omega; \mathbb{R}^m)$ is the weak solution of
$$
 \left\{
 \aligned
 \mathcal{L}_\varep (v_\varep)& =0 & \quad &\text{ in } \Omega,\\
v_\varep& =u_\varep -v_0 -\varep \chi_T(x/\varep) \nabla v_0&\quad  &\text{ on }
\partial\Omega.
\endaligned
\right.
$$
Then, for any $\sigma \in (0,1)$,
\begin{equation}\label{H-1-1}
\| w_\varep \|_{H^1(\Omega)}
\le C_\sigma \big\{ \Theta_\sigma (T) +\langle |\psi-\nabla \chi_T|\rangle \big\} 
\left\{ \|\nabla^2 v_0\|_{L^2(\Omega)} +\| \nabla v_0\|_{L^2(\Omega)} \right\},
\end{equation}
where $\psi =\big(\psi_{ij}^{\alpha\beta}\big)$ is defined by
(\ref{homo-equation})
 and $C_\sigma$ depends only on $\sigma$, $A$, and $\Omega$.
\end{lemma}

\begin{proof}
The proof is similar to that of Theorem 7.3 in \cite{Shen-2014}, 
where $v_0$ is taken to be $u_0$.
A direct computation shows that
$$
\mathcal{L}_\varep (w_\varep)
=-\text{\rm div} \big(B_T (x/\varep)\nabla v_0\big) + \varep\, \text{\rm div} \big\{ A(x/\varep)\chi_T (x/\varep)
\nabla^2 v_0\big\},
$$
where $B_T (y) =\big( b_{T, ij}^{\alpha\beta} \big)$ is given by
\begin{equation}\label{B-T}
b_{T, ij}^{\alpha\beta} (y)
=\widehat{a}_{ij}^{\alpha\beta} -a_{ij}^{\alpha\beta} (y)
-a_{ik}^{\alpha\gamma} (y) \frac{\partial }{\partial y_k} \left\{ \chi_{T, j}^{\gamma\beta} (y)\right\}.
\end{equation}
Since $w_\varep =0$ on $\partial\Omega$, it follows that
\begin{equation}\label{H-0-1}
c\int_\Omega |\nabla w_\varep|^2
\le \left|\int_\Omega \text{\rm div} \big( B_T (x/\varep)\nabla v_0\big) \cdot w_\varep\, dx\right|
+\int_\Omega |\varep\chi_T (x/\varep) |\, |\nabla^2 v_0|\, |\nabla w_\varep|\, dx.
\end{equation}
Thus, it suffices to show that the right hand side of (\ref{H-0-1}) is bounded by
$$
C_\sigma \big\{ \Theta_\sigma (T) +\langle |\psi-\nabla \chi_T|\rangle \big\}
\left\{ \|\nabla^2 v_0\|_{L^2(\Omega) }+\|\nabla v_0\|_{L^2(\Omega)} \right\}
\| w_\varep\|_{H^1(\Omega)}
$$
for any $\sigma\in (0,1)$.
By (\ref{ac-bound}) and Cauchy inequality,
the second integral in the right hand side of (\ref{H-0-1}) is bounded by
$$
 C_\sigma\, \Theta_\sigma (T) \, \|\nabla^2 v_0\|_{L^2(\Omega)} \|\nabla w_\varep\|_{L^2(\Omega)}.
$$
The estimate of the first integral is much more delicate
and is done by the same argument as in the proof of Theorem 7.3 in \cite{Shen-2014}.
The key idea is to solve the equation 
\begin{equation}\label{H}
-\Delta H +T^{-2} H=B_T  -\langle B_T \rangle \quad \text{ in } \mathbb{R}^d,
\end{equation}
and show that there exists a solution $H=H_T=\left( h_{ij}^{\alpha\beta} \right)\in W^{2,2}_{loc} (\mathbb{R}^d)$ 
satisfying
\begin{equation}\label{H-1}
\left\{
\aligned
T^{-2} \| H\|_\infty & \le C \, \Theta_1 (T), \\
T^{-1} \|\nabla H\|_\infty & \le C_\sigma\, \Theta_\sigma (T),
\endaligned
\right.
\end{equation}
and
\begin{equation}\label{H-2}
\bigg\|\nabla \frac{\partial h_{ij}^{\alpha\beta}}{\partial x_i} \bigg\|_\infty  \le C_\sigma \, \Theta_\sigma (T)
\end{equation}
for any $\sigma \in (0,1)$ (the index $i$ in (\ref{H-2}) is summed from $1$ to $d$).
We omit the details.
A similar approach will be used in the proof of Lemma \ref{NP-rate-lemma-1}.
\end{proof}

\begin{lemma}\label{lemma-H-2}
Let $\Omega$ be a bounded $C^{1, \alpha}$ domain in $\mathbb{R}^d$ for some $\alpha>0$.
Let $u_\varep\in H^1(\Omega; \mathbb{R}^m)$ $(\varep\ge 0)$ be the weak solution of (\ref{DP-1}).
Then, for any $\sigma, \delta\in (0,1)$,
\begin{equation}\label{H-2-1}
\aligned
\| u_\varep -u_0\|_{L^2(\Omega)}
 & \le C \big\{ \Theta_\sigma (T) +\langle |\psi-\nabla \chi_T|\rangle \big\}
\big\{ \|\nabla^2 v_0\|_{L^2(\Omega)} +\| \nabla v_0\|_{L^2(\Omega)} \big\}\\
&\qquad \qquad +C \, \| f- v_0 \|_{C^\delta (\partial\Omega)} +
C \big[\Theta_\sigma (T)\big]^{1-\delta} \| \nabla v_0 \|_{C^\delta (\partial\Omega)},
\endaligned
\end{equation}
where $v_0 \in W^{2,2}(\Omega; \mathbb{R}^m)
\cap C^{1, \delta} (\overline{\Omega}; \mathbb{R}^m)$ and
$\mathcal{L}_0 (v_0)=F$ in $\Omega$.
The constant $C$ depends only on $\delta$, $\sigma$, $A$, and $\Omega$.
\end{lemma}

\begin{proof}
Let $v_\varep$ and $w_\varep$ be defined as in Lemma \ref{lemma-H-1}.
Then
\begin{equation}\label{H-2-2}
\| u_\varep -u_0\|_{L^2(\Omega)}
\le \| w_\varep\|_{L^2(\Omega)}
+\| v_0 -u_0\|_{L^2(\Omega)}
+\varep \|\chi_T\|_\infty \|\nabla v_0\|_{L^2(\Omega)}
+\| v_\varep\|_{L^2(\Omega)}.
\end{equation}
In view of (\ref{H-1-1}) we only need to handle the last three terms in the right hand side of (\ref{H-2-2}).

First, since $\mathcal{L}_0 (v_0 -u_0)=0$ in $\Omega$ and $\Omega$ is $C^{1,\alpha}$, 
we obtain 
\begin{equation}\label{H-2-3}
\| v_0 -u_0\|_{L^2(\Omega)}
\le C\, \| v_0 -f\|_{L^2(\partial\Omega)}.
\end{equation}
Next, we note that
$$
\varep \, \|\chi_T\|_\infty \|\nabla v_0\|_{L^2(\Omega)}
=T^{-1}  \|\chi_T\|_\infty \|\nabla v_0\|_{L^2(\Omega)}
\le C_\sigma \Theta_\sigma (T)\| \nabla v_0\|_{L^2(\Omega)}.
$$

Finally, recall that $\mathcal{L}_\varep (v_\varep)=0$ in $\Omega$ and 
$v_\varep =f-v_0 -\varep \chi_T(x/\varep) \nabla v_0$ on $\partial\Omega$.
Since $\Omega$ is $C^{1, \alpha}$,
we may use the H\"older estimates (\ref{H-estimate}) to obtain
\begin{equation}
\aligned
\| v_\varep \|_{L^2(\Omega)}
&\le C\, \| v_\varep \|_{C^{\delta_1} (\partial\Omega)}\\
&\le C\, \| f-v_0\|_{C^{\delta_1}(\partial\Omega)} 
+C  \|\varep  \chi_T (x/\varep)\|_{C^{\delta_1} (\partial\Omega)}
\| \nabla v_0 \|_{C^{\delta_1} (\partial\Omega)}
\endaligned
\end{equation}
for any $\delta_1\in (0,1)$.
Note that $\|\varep \chi_T (x/\varep)\|_\infty \le C_\sigma \Theta_\sigma (T)$ and
$\|\varep \chi_T (x/\varep)\|_{C^{0, \delta}} \le C_\delta$.
By interpolation this implies that 
$$
\|\varep \chi_T(x/\varep)\|_{C^{\delta_1}} \le C\, \big[\Theta_\sigma (T)\big]^{1- \delta_2}
$$
for any $\sigma \in (0,1)$ and $0<\delta_1<\delta_2<1$.
As a result, we see that
$$
\| v_\varep\|_{L^2(\Omega)}
\le C\, \| f-v_0\|_{C^\delta(\partial\Omega)}
+C\, \big[ \Theta_\sigma (T)\big]^{1-\delta} \| \nabla v_0\|_{C^\delta(\partial\Omega)}
$$
for any $\sigma, \delta \in (0,1)$.
The proof is now complete.
\end{proof}

\begin{remark}
{\rm 
If we let $v_0=u_0$ in Lemma \ref{lemma-H-2}, then
\begin{equation}\label{remark-H-1}
\aligned
\| u_\varep -u_0\|_{L^2(\Omega)}
&\le C\, \omega (\varep)\, 
\| u_0 \|_{W^{2, p} (\Omega)},
\endaligned
\end{equation}
where $p>d$ and
\begin{equation}\label{omega}
\omega (\varep)=\omega_\sigma (\varep)
=\big[\Theta_1 (\varep^{-1})\big]^{\sigma}
+\sup_{T\ge \varep^{-1}} \langle |\psi -\nabla \chi_T |\rangle.
\end{equation}
Here we have used the observation $\Theta_\sigma (T) \le C_\sigma \big[\Theta_1(T)\big]^\sigma$
as well as Sobolev imbedding $\|\nabla u_0\|_{C^\delta(\Omega)} \le C\, \| u_0\|_{W^{2, p}(\Omega)}$
for $p>d$ and $\delta=1-(d/p)$.
Note that $\omega (\varep)$ is a nondecreasing continuous function on $(0,1]$ and
$\omega (0+)=0$.
}
\end{remark}

 Estimate (\ref{remark-H-1}) is one of the main results proved in \cite{Shen-2014}.
 In the periodic setting it gives a near optimal convergence rate of $O(\varep^\gamma)$ for any 
 $\gamma\in (0,1)$.
 However, since $\Omega$ is only assumed to be $C^{1,\alpha}$, the $W^{2,p}$ norm in (\ref{remark-H-1})
 is not convenient in some applications.
 Our next theorem is an attempt to resolve this issue
 (see \cite{KLS2} for analogous results in the periodic setting).
 For simplicity we assume that $F=0$.
 
 \begin{theorem}\label{rate-theorem-1}
 Suppose that $A(y)$ is uniformly almost-periodic in $\mathbb{R}^d$ and satisfies (\ref{ellipticity}).
 Let $\Omega$ be a bounded $C^{1, \alpha}$ domain in $\mathbb{R}^d$ for some $\alpha>0$.
 Let $u_\varep $ $ (\varep\ge 0)$ be the weak solution of Dirichlet problem:
 $\mathcal{L}_\varep (u_\varep)=0$ in $\Omega$ and $u_\varep=f$ on $\partial\Omega$.
 Then, for any $\delta \in (0, \alpha)$,
 \begin{equation}\label{rate-2}
 \| u_\varep -u_0\|_{L^2(\Omega)}
 \le C  \, \big[ \omega (\varep)\big]^{2/3} \, 
 \| f\|_{C^{1,\delta}(\partial\Omega)},
 \end{equation}
 where $\omega (\varep)=\omega_\sigma (\varep)$ is defined by (\ref{omega}) and
 $C$ depends only on  $\delta$, $\sigma$, $A$, and $\Omega$.
 \end{theorem}
 
 \begin{proof}
 We begin by constructing a family of bounded $C^{1, \alpha}$ domains $\{ \Omega_s: s\in (0,1/2)\}$
 such that  (1) $\Omega\subset \Omega_s$,
 (2) for each $s\in (0,1/2)$,  there is a $C^{1, \alpha}$
 diffeomorphism $\Lambda_s: \partial\Omega\to \partial \Omega_s$ with uniform bounds, 
 and (3) $|x-\Lambda_s (x)|\approx \text{dist}(x, \partial\Omega_s)
 \approx s$ for every $x\in \partial\Omega$.
 The constants $C$ in the estimates below do not depend on $s$.
 
 Next, let $f_s (x) =f (\Lambda^{-1}_s (x))$ for $x\in \partial\Omega_s$ and
 $v=v_s$ be the solution of  Dirichlet problem:
 $\mathcal{L}_0 (v)=0$ in $\Omega_s$ and $v=f_s$ on $\partial\Omega_s$.
  We will show that
 \begin{equation}\label{H-4-1}
 \| u_\varep- u_0\|_{L^2(\Omega)}
 \le C \left\{ s + s^{-\frac{1}{2}} \omega (\varep)\right\} \| f\|_{C^{1,\delta} (\partial\Omega)}.
 \end{equation}
  Estimate (\ref{rate-2}) follows from (\ref{H-4-1}) by choosing
 $s\in (0,1/2)$ so that $s^{3/2} =c\, \omega (\varep)$.
 
 To see (\ref{H-4-1}), we use Lemma \ref{lemma-H-2} to obtain
 \begin{equation}\label{H-4-2}
 \| u_\varep -u_0\|_{L^2(\Omega)}
  \le C\, \omega (\varep) 
 \left\{ \|\nabla^2 v\|_{L^2(\Omega)}
+\|\nabla v\|_{C^{\delta} (\Omega)} \right\}
 +C\, \| f-v\|_{C^\delta(\partial\Omega)}
 \end{equation}
 for any $\delta\in (0,\alpha)$.
 Since $\mathcal{L}_0 (v)=0$ in $\Omega_s$ and $\Omega_s$ is $C^{1, \alpha}$,
 \begin{equation}\label{H-4-3}
 \| f-v\|_{C^\delta (\partial\Omega)} 
 \le C\, s \, \|\nabla v \|_{C^{ \delta} (\Omega_s)}
 \le C\, s\, \| f_s\|_{C^{1,\delta}(\partial\Omega_s)}
 \le C\, s\, \| f\|_{C^{1, \delta}(\partial\Omega)}.
 \end{equation}
  By the interior estimates for $\mathcal{L}_0$ and the fact that
 $\Omega \subset \{ x\in \Omega_s: \text{dist}(x, \partial\Omega_s)\ge c\, s\}$,
  it is not hard to see that
 $$
 \aligned
 \int_\Omega |\nabla^2 v|^2\, dx 
 &\le C\, \|\nabla v\|_{L^\infty(\Omega_s)}^2 \int_\Omega \frac{dx}{\big[ \text{dist} (x, \partial\Omega_s)\big]^2}\\
& \le C\, s^{-1} \|\nabla v\|_{L^\infty(\Omega_s)}^2
\le C\, s^{-1}\| f\|^2_{C^{1, \delta}(\partial\Omega)}.
\endaligned
 $$
 This, together with (\ref{H-4-2})-(\ref{H-4-3}) and
 the estimate $\| \nabla v\|_{C^{ \delta}(\Omega)} \le C\, \| f\|_{C^{1,\delta}(\partial\Omega)}$,
  yields (\ref{H-4-1}).
 \end{proof}
 

\subsection{Convergence rates: Neumann boundary conditions}

In this subsection we establish estimates on convergence rates for the Neumann problem
(\ref{NP-1}) under an additional assumption  that
\begin{equation}\label{ac-Lip}
\sup_{T\ge 1} \| \nabla \chi_T \|_{L^\infty(\mathbb{R}^d)} \le C_0 <\infty.
\end{equation}
This condition follows from the uniform interior Lipschitz estimates (see Remark \ref{ac-remark}).
In particular, it holds under the assumptions on $A$ in Theorem \ref{main-theorem-Lip}.

\begin{lemma}\label{NP-rate-lemma-1}
Suppose that $A$ is uniformly almost-periodic and satisfies (\ref{ellipticity}).
Also assume that the condition (\ref{ac-Lip}) holds.
Let $\Omega$ be a bounded Lipschitz domain in $\mathbb{R}^d$.
Let
\begin{equation}
\left\{
\aligned
\mathcal{L}_\varep (u_\varep) & =F &\quad &\text{ in } \Omega,\\
\frac{\partial u_\varep}{\partial \nu_\varep}
&=g&\quad &\text{ on } \partial\Omega,
\endaligned
\right.
\quad \text{ and } \quad 
\left\{
\aligned
\mathcal{L}_0 (v_0) & =F &\quad &\text{ in } \Omega,\\
\frac{\partial v_0}{\partial \nu_0}
&=g_0&\quad &\text{ on } \partial\Omega,
\endaligned
\right.
\end{equation}
where $F\in L^2(\Omega; \mathbb{R}^m)$ and
$g, g_0\in L^2(\partial\Omega;\mathbb{R}^m)$.
Suppose $\int_\Omega u_\varep =\int_\Omega v_0 $.
Then, for any $\sigma\in (0,1)$,
\begin{equation}\label{NP-rate-1}
\aligned
\| u_\varep -v_0\|_{L^2(\Omega)}
& \le C\, \| g-g_0 \|_{H^{-1/2}(\partial\Omega)}+
 C\, \big\{\Theta_\sigma (T)
+\langle|\psi-\nabla \chi_T |\rangle\big\}\|\nabla^2 v_0\|_{L^2(\Omega)} \\
&\qquad\qquad\qquad
+C \left\{ \big[ \Theta_\sigma (T)\big]^{\frac12}
+\langle|\psi-\nabla \chi_T |\rangle\right\}
\|(\nabla v_0)^*\|_{L^2(\partial \Omega)},
\endaligned
\end{equation}
where $T=\varep^{-1}$ and $(\nabla v_0)^*$ denotes the non-tangential
maximal function of $\nabla v_0$.
The constant $C$ in (\ref{NP-rate-1}) depends only on $\sigma$, $A$, and $\Omega$.
\end{lemma}

\begin{proof}
As in the case of Dirichlet boundary condition, we consider 
$$
w_\varep =u_\varep -v_0 -\varep \chi_T (x/\varep) \nabla v_0,
$$
where $T=\varep^{-1}$.
We will show that
\begin{equation}\label{NP-rate-2}
\aligned
\|\nabla w_\varep\|_{L^2(\Omega)}
& \le C\, \| g-g_0 \|_{H^{-1/2}(\partial\Omega)}+
C \big\{\Theta_\sigma (T)
+\langle|\psi-\nabla \chi_T |\rangle\big\}\|\nabla^2 v_0\|_{L^2(\Omega)} \\
&\qquad\qquad\qquad
+C \left\{ \big[ \Theta_\sigma (T)\big]^{\frac12}
+\langle|\psi-\nabla \chi_T |\rangle\right\}
\|(\nabla v_0)^*\|_{L^2(\partial \Omega)}.
\endaligned
\end{equation}
Since $|\int_\Omega w_\varep|\le  \varep \|\chi_T\|_\infty \|\nabla v_0\|_{L^1(\Omega)}
\le C\, \Theta_\sigma (T) \|\nabla v_0\|_{L^1(\Omega)}$,
the estimate (\ref{NP-rate-1}) follows from (\ref{NP-rate-2})
by Poincar\'e inequality.

To prove (\ref{NP-rate-2}), we observe that
\begin{equation}\label{NP-rate-3}
\aligned
&\int_\Omega \nabla w_\varep \cdot A(x/\varep)\nabla w_\varep\, dx\\
&=<w_\varep, g-g_0>
-\int_\Omega \nabla w_\varep \cdot B_T(x/\varep) \nabla v_0
-\int_\Omega \nabla w_\varep\cdot A(x/\varep)\, \varep \chi_T(x/\varep)\nabla^2 v_0,
\endaligned
\end{equation}
where $B_T(y)=\widehat{A} -A(y)-A(y)\nabla \chi_T (y)$, and we have used the fact
$$
\int_\Omega \nabla w_\varep\cdot \big( A(x/\varep)\nabla u_\varep
-\widehat{A}\nabla v_0 \big)
=<w_\varep, g-g_0>.
$$
Since $\int_{\partial\Omega} (g-g_0)=0$,
$$
\aligned
|< w_\varep, g-g_0>| &\le \| g-g_0\|_{H^{-1/2}(\partial\Omega)}
\| w_\varep -E\|_{H^{1/2}(\partial\Omega)}\\
&\le C\, \| g-g_0\|_{H^{-1/2}(\partial\Omega)}
\|\nabla w_\varep\|_{L^2(\Omega)},
\endaligned
$$
where $E=\average_\Omega w_\varep$.
Also, the last term in the right hand of (\ref{NP-rate-3}) is bounded by 
$$
C_\sigma \Theta_\sigma (T)\|\nabla w_\varep\|_{L^2(\Omega)}
\|\nabla^2 v_0\|_{L^2(\Omega)} .
$$
Furthermore, since 
$
|\langle B_T\rangle|\le C \, \langle |\psi-\nabla \chi_T| \rangle,
$
in view of (\ref{NP-rate-3}), 
it suffices to show that
\begin{equation}\label{NP-rate-4}
\aligned
& \left| \int_\Omega \nabla w_\varep \cdot \left\{ B_T (x/\varep)-\langle B_T \rangle \right\} \nabla v_0\right|\\
&\qquad \le C\, \| \nabla w_\varep\|_{L^2(\Omega)} 
 \left\{\Theta_\sigma (T)
+\langle|\psi-\nabla \chi_T |\rangle\right\}\|\nabla^2 v_0\|_{L^2(\Omega)} \\
&\qquad\qquad\qquad\qquad
+C \, \|\nabla w_\varep\|_{L^2(\Omega)}
\big[ \Theta_\sigma (T)\big]^{\frac12}
\|(\nabla v_0)^*\|_{L^2(\partial \Omega)}.
\endaligned
\end{equation}
This will be done by using a line of argument similar to that used in the proof of
Theorem 7.3 in \cite{Shen-2014} as well as  in the proof of Lemma \ref{lemma-H-1}.

Let $H=H_T\in W^{2,2}_{loc} (\mathbb{R^d})$ be a solution of (\ref{H}) that satisfies (\ref{H-1})-(\ref{H-2}).
In view of the first estimate in (\ref{H-1}),
it suffices to prove (\ref{NP-rate-4}) with $B_T(x/\varep)-\langle B_T\rangle$
replaced by $\Delta H(x/\varep)$.
Let $\varphi=\varphi_\delta \in C_0^\infty(\mathbb{R}^d)$ be a cut-off function
such that $0\le \varphi\le 1$, $\varphi (x)=1$ if dist$(x, \partial\Omega)\ge 2c\delta$,
$\varphi (x)=0$ if dist$(x, \partial\Omega)\le c\delta$, and $|\nabla\varphi|\le C\delta^{-1}$,
where $\delta\in (0,1)$ is to be determined.
A direct computation shows  that for each $1\le j\le d$ and $1\le \beta\le m$,
$$
\aligned
\frac{\partial w_\varep^\alpha}{\partial x_i} \cdot \Delta h_{ij}^{\alpha\beta} (x/\varep)
&=\frac{\partial }{\partial x_k}
\left\{ \frac{\partial w_\varep^\alpha}{\partial x_i} \cdot \varep
\frac{\partial h_{ij}^{\alpha\beta}}{\partial x_k} (x/\varep) \right\}
-\frac{\partial }{\partial x_i}
\left\{ \frac{\partial w_\varep^\alpha}{\partial x_k} \cdot \varep
\frac{\partial h_{ij}^{\alpha\beta}}{\partial x_k} (x/\varep) \right\}\\
&\qquad\qquad\qquad
+\frac{\partial w_\varep^\alpha}{\partial x_k} \cdot 
\frac{\partial^2 h_{ij}^{\alpha\beta}}{\partial x_i \partial x_k} (x/\varep),
\endaligned
$$
where the summation convention is used.
It follows that
\begin{equation}\label{NP-rate-5}
\aligned
&\left |\int_\Omega \nabla w_\varep \cdot \Delta H (x/\varep) (\nabla v_0 )\varphi\right|
\le C\, \varep  \int_\Omega |\nabla w_\varep|\, |\nabla H(x/\varep)|\, |\nabla \big( (\nabla v_0)\varphi\big)|\\
& \qquad\qquad \qquad\qquad\qquad\qquad\qquad
+C \int_\Omega \left|\frac{\partial w^\alpha_\varep}{\partial x_k}\right|
\, \left| \frac{\partial h_{ij}^{\alpha\beta}}{\partial x_i\partial x_k} (x/\varep)\right|\,
\left|\frac{\partial v_0^\beta}{\partial x_j}\right|\varphi\\
&\le C\, \Theta_\sigma (T) \, \| \nabla w_\varep\|_{L^2(\Omega)}
\left\{ \|\nabla^2 v_0\|_{L^2(\Omega)}
+\delta^{-1/2} \|(\nabla v_0)^*\|_{L^2(\partial \Omega)} \right\},
\endaligned
\end{equation}
where we have used estimates (\ref{H-1}) and (\ref{H-2}) as well as 
the observation 
$$
\|(\nabla v_0)(\nabla \varphi)\|_{L^2(\Omega)}
\le C\delta^{-1/2} \|(\nabla v_0)^*\|_{L^2(\partial\Omega)}.
$$

Finally, using the condition  (\ref{ac-Lip}), we see that
$$
\|\Delta H\|_\infty \le T^{-2} \| H\|_\infty +2 \| B_T\|_\infty\le 
C+ C\, \|\nabla \chi_T\|_\infty \le C.
$$
Hence,
$$
\aligned
\left |\int_\Omega \nabla w_\varep \cdot \Delta H (x/\varep) (\nabla v_0 )(1-\varphi) \right|
&\le C\, \|\Delta H\|_\infty \|\nabla w_\varep\|_{L^2(\Omega)}
\left\{ \int_{\substack{x\in \Omega\\ \text{dist} (x, \partial\Omega)\le 2c\delta }}
|\nabla v_0|^2\right\}^{1/2}\\
&\le C\, \delta^{1/2} \|\nabla w_\varep \|_{L^2(\Omega)} \|(\nabla v_0)^*\|_{L^2(\partial \Omega)}.
\endaligned
$$
By choosing $\delta=c \, \Theta_\sigma (T)$,
this, together with (\ref{NP-rate-5}), completes the proof.
\end{proof}

\begin{remark}
{\rm
Let $u_\varep$ $(\varep\ge 0)$ be the weak solution of (\ref{NP-1}) in a bounded Lipschitz domain
$\Omega$. It follows from Lemma \ref{NP-rate-lemma-1} that
\begin{equation}\label{NP-rate-7}
\aligned
\| u_\varep -u_0\|_{L^2(\Omega)}
&\le
C\, \big\{\Theta_\sigma (T)
+\langle|\psi-\nabla \chi_T |\rangle\big\}\|\nabla^2 u_0\|_{L^2(\Omega)} \\
&\qquad\qquad
+C \left\{ \big[ \Theta_\sigma (T)\big]^{\frac12}
+\langle|\psi-\nabla \chi_T |\rangle\right\}
\|(\nabla u_0)^*\|_{L^2(\partial \Omega)}.
\endaligned
\end{equation}
This estimate is not sharp in the periodic setting.
It only gives $\| u_\varep -u_0\|_{L^2(\Omega)} =  O( \varep^\gamma)$
for any $0<\gamma<(1/2)$.
}
\end{remark}

\begin{theorem}\label{NP-rate-theorem-2}
Suppose that $A$ satisfies the same condition as in Lemma \ref{NP-rate-lemma-1}.
Let $\Omega$ be a bounded $C^{1, \alpha}$ domain in $\mathbb{R}^d$ with connected boundary
for some $\alpha>0$.
Let 
\begin{equation}
\left\{
\aligned
\mathcal{L}_\varep (u_\varep) & =0 &\quad &\text{ in } \Omega,\\
\frac{\partial u_\varep}{\partial \nu_\varep}
&=g&\quad &\text{ on } \partial\Omega,
\endaligned
\right.
\quad \text{ and } \quad 
\left\{
\aligned
\mathcal{L}_0 (u_0) & =0 &\quad &\text{ in } \Omega,\\
\frac{\partial u_0}{\partial \nu_0}
&=g&\quad &\text{ on } \partial\Omega,
\endaligned
\right.
\end{equation}
where $g\in L^2(\partial \Omega; \mathbb{R}^m)$.
Assume $\int_\Omega u_\varep =\int_\Omega u_0$.
Then
\begin{equation}\label{NP-rate-00}
\| u_\varep -u_0 \|_{L^2(\Omega)}
\le C\, \big \{ \Theta_\sigma (T)+ \langle |\psi-\nabla \chi_T|\rangle\big\}^{1/2}
   \| g\|_{L^2(\partial\Omega)},
\end{equation}
where $T=\varep^{-1}$, 
$\sigma \in (0,1)$, and $C$ depends only on $\sigma$, $A$, and $\Omega$.
\end{theorem}

\begin{proof}
We begin by constructing a family of $C^{1, \alpha}$ domains $\{ \Omega_t: \, t\in (0, 1)\}$
with the property that
(1) $\Omega \subset \Omega_t$, (2) theres exist $C^{1, \alpha}$ diffeomorphisms 
$\Lambda_t: \partial \Omega\to \partial\Omega_t$ with uniform bounds such that
dist$(x, \Lambda_t (x)) \approx \text{dist} (x, \partial\Omega_t )\approx t$
for any $x\in \partial\Omega$.
Let $v=v_t$ be the weak solution to Dirichlet problem:
$\mathcal{L}_0 (v)=0$ in $\Omega_t$ and $v= f_t$ on $\partial\Omega$,
where $f_t(x)=u_0 (\Lambda^{-1}_t (x) )$ for $x\in \partial\Omega_t$.

Next we use the non-tangential maximal function estimates 
$$
\left\{
\aligned
& \| (w)^*\|_{L^2(\partial\Omega_t)} 
\le C\, \| w\|_{L^2(\partial\Omega_t)}, \quad 
\|(\nabla w)^* \|_{L^2 (\partial\Omega_t)} \le C\, \|\nabla_{tan} w \|_{L^2(\partial \Omega_t)},\\
& \|(\nabla w)^* \|_{L^2(\partial\Omega_t)}
\le C\, \|\frac{\partial w}{\partial \nu_0} \|_{L^2(\partial\Omega_t)}
\endaligned
\right.
$$
for the
$L^2$ Dirichlet and Neumann problems for the system $\mathcal{L}_0 (w)=0$ in $C^{1, \alpha}$
domains to control $\| u_0 -v\|_{L^2(\Omega)}$.
This gives
\begin{equation}\label{NP-rate-10}
\aligned
\| u_0 -v\|_{L^2(\Omega)}
&\le C\, \| u_0 -v\|_{L^2(\partial\Omega)}
\le C\,  t\, \| (\nabla v)^*\|_{L^2(\partial\Omega_t)}\\
&\le C\, t\, \|\nabla_{tan} v \|_{L^2(\partial\Omega_t)}
\le C\, t\, \|\nabla u_0\|_{L^2(\partial\Omega)}\\
&\le C\, t\, \| g\|_{L^2(\partial\Omega)}.
\endaligned
\end{equation}
 To handle $\| u_\varep -v\|_{L^2(\Omega)}$,
we use Lemma \ref{NP-rate-lemma-1} to obtain
\begin{equation}\label{NP-rate-20}
\aligned
\| u_\varep -v\|_{L^2(\Omega)}
&\le C\, \| g- g_0\|_{H^{-1/2} (\partial\Omega)}
+C \left\{ \Theta_\sigma (T) +\langle |\psi-\nabla \chi_T |\rangle \right\} \| \nabla^2 v\|_{L^2(\Omega)}\\
&\qquad\qquad\qquad
+ C  \left\{ \big[\Theta_\sigma (T) \big]^{1/2}
+\langle |\psi-\nabla \chi_T |\rangle \right\}\|(\nabla v)^*\|_{L^2(\partial\Omega)},
\endaligned
\end{equation}
where $g_0 =\frac{\partial v}{\partial\nu_0}$.
Since $\mathcal{L}_0 (u_0 -v)=0$ in $\Omega$, we see that
\begin{equation}\label{NP-rate-30}
\aligned
\| g-g_0\|_{H^{-1/2}(\partial\Omega)}
&\le C\, \| u_0 -v\|_{H^1(\Omega)} \le C\, \| u_0 -v\|_{H^{1/2}(\partial\Omega)}\\
&\le C\, \|u_0 -v\|_{L^2(\partial\Omega)}^{1/2} \| u_0 -v\|_{H^1(\partial\Omega)}^{1/2}\\
&\le C\, t^{1/2} \|(\nabla v)^*\|_{L^2(\partial\Omega)}
\le C\, t^{1/2}\, \| g\|_{L^2(\partial\Omega)}.
\endaligned
\end{equation}
Also, since $\mathcal{L}_0 (v)=0$ in $\Omega_t$,
by the square function estimate \cite{Dahlberg-Kenig-Pipher-Verchota},
$$
\left\{ \int_{\Omega_t} |\nabla^2 v (x) |^2 \text{\rm dist} (x, \partial\Omega_t)\, dx \right\}^{1/2}
\le C\, \| v\|_{H^1(\partial\Omega_t)},
$$
we obtain 
\begin{equation}\label{NP-rate-40}
\| \nabla^2 v\|_{L^2(\Omega)} \le C\, t^{-1/2}\, \| g\|_{L^2(\partial\Omega)}.
\end{equation}
In view of (\ref{NP-rate-10})-(\ref{NP-rate-40}) we have proved that
$$
\aligned
\| u_\varep -u_0\|_{L^2(\Omega)}
&\le C\, t^{-1/2} \left\{ \Theta_\sigma (T) +\langle |\psi-\nabla \chi_T|\rangle\right\}  \| g\|_{L^2(\partial\Omega)}\\
&\qquad + C\,  \left\{ \big[\Theta_\sigma (T) \big]^{1/2}
+\langle |\psi-\nabla \chi_T|\rangle\right\}  \| g\|_{L^2(\partial\Omega)}
+ C\, t^{1/2} \| g\|_{L^2(\partial\Omega)}.
\endaligned
$$
Finally, the estimate (\ref{NP-rate-00}) follows by choosing 
$t=c\big\{ \Theta_\sigma (T) +\langle |\psi-\nabla \chi_T|\rangle\big\} $.
\end{proof}



\section{A general scheme for Lipschitz estimates at large scale}
\setcounter{equation}{0}

In this section we present a general scheme for proving
Lipschitz estimates at large scale in homogenization.
As we pointed out in Introduction,
the scheme, which was motivated by the compactness method in \cite{AL-1987},
was recently formulated by the first author and C. Smart in \cite{Armstrong-Smart-2014}.
The $L^2$ version of the scheme in this section is a slight
variation of the one given in \cite{Armstrong-Smart-2014}.

\begin{lemma}\label{main-lemma-1}
Let $\{F_0, F_1, \dots, F_\ell\}$ and $\{ p_0, p_1, \dots, p_\ell\}$ be two sequences of nonnegative numbers.
Suppose that for $0\le j\le \ell-1$,
\begin{equation}\label{M-L-1}
p_{j+1} \le p_j + C_0 \max \big\{ F_j, F_{j+1} \big\},
\end{equation}
and for $1\le j\le \ell-1$,
\begin{equation}\label{M-L-2}
F_{j+1} \le \frac12 F_j + \eta_j K +
\eta_j \max \big\{ p_0, \dots, p_{j-1} \big\} +\eta_j \max \big\{ F_0, \dots, F_{j-1}\big \},
\end{equation}
where $K\ge 0$,
$0\le \eta_1 \le \eta_2 \le \cdots \le \eta_{\ell-1}=\eta_\ell$ and $\eta_1 +\eta_2 +\cdots \eta_{\ell} \le C_1$.
Then for $1\le j\le \ell$,
\begin{equation}\label{M-L-3}
p_j \le C \,( K+ p_0 +F_0 +F_1),
\end{equation}
\begin{equation}\label{M-L-4}
F_j \le C\, (2^{-j} + \eta_j ) (K+ p_0 +F_0 +F_1),
\end{equation}
where $C$ depends only on $C_0$ and $C_1$.
\end{lemma}

\begin{proof}
The proof of this lemma is essentially contained in the proof of \cite[Lemma 5.1]{Armstrong-Smart-2014}.
We provide a proof here for the sake of completeness.

By considering $\widetilde{p}_j =p_j +K$, we may assume that $K=0$.
Let 
\begin{equation}
T_j =F_j -2 \eta_j \max \big\{ p_0, \dots, p_{j-1}\big\}  -2 \eta_j \max \big\{ F_0, \cdots, F_{j-1} \big\}.
\end{equation}
Note that 
$$
\aligned
T_{j+1} &= F_{j+1} -2 \eta_{j+1} \max\big\{ p_0, \dots, p_j\big\} -2 \eta_{j+1} \max\big \{ F_0, \dots, F_j\big\}\\
&\le \frac12 F_j + \eta_j \max \big\{ p_0, \dots, p_{j-1} \big\} +\eta_j \max \big\{ F_0, \dots, F_{j-1} \big\}\\
&\qquad\qquad\qquad
 -2 \eta_{j+1} \max\big\{ p_0, \dots, p_j\big\} -2 \eta_{j+1} \max\big \{ F_0, \dots, F_j\big\}\\
 &\le  \frac12 F_j
 +(\eta_j -2\eta_{j+1})\max \big\{p_0, \dots, p_{j-1}\big\}
 +(\eta_j -2\eta_{j+1})\max \big\{ F_0, \cdots, F_{j-1}\big\},
\endaligned
$$
where we have used (\ref{M-L-2}) for the first inequality.
Since $\eta_j -2\eta_{j+1} \le -\eta_j$, we obtain $T_{j+1} \le (1/2) T_j $ for $1\le j\le\ell-1$.
It follows that $T_j \le (1/2)^{j-1} T_1\le (1/2)^{j-1} F_1$. Hence,
\begin{equation}\label{M-L-5}
F_j \le (1/2)^{j-1} F_1
+2 \eta_j \max \big\{ p_0, \dots, p_{j-1}\big\}  +2 \eta_j \max \big\{ F_0, \cdots, F_{j-1} \big\}.
\end{equation}

Next we will show that for $0\le j\le \ell$,
\begin{equation}\label{M-L-7}
F_j \le C_2 \bigg\{ (2^{-j} +\eta_j )(F_0+F_1) +\eta_j \max \big\{ p_0, \dots, p_{j-1} \big\}\bigg\},
\end{equation}
where $C_2 $ depends only on $C_1$. To prove (\ref{M-L-7}), we claim that
for $1\le j\le \ell$,
\begin{equation}\label{M-L-9}
F_j \le 2 (1+2\eta_1) \cdots (1+2\eta_{j} )
\bigg\{ (2^{-j} +\eta_j)(F_0+   F_1) + \eta_j \max \big\{ p_0, \dots, p_{j-1} \big\} \bigg\}.
\end{equation}
Since $\eta_1 +\eta_2 +\cdots \eta_{\ell} \le C_1$, one may use the inequality
$\ln (1+x)\le x$ for $x\ge 0$ to see that
\begin{equation}\label{M-L-10}
(1+C\eta_1) \cdots (1+C\eta_{\ell} )\le e^{CC_1}.
\end{equation}
As a result, estimate (\ref{M-L-7}) follows from (\ref{M-L-9}).

Estimate (\ref{M-L-9}) is proved by induction, using (\ref{M-L-5}).
Indeed, suppose (\ref{M-L-9}) holds for $1\le j\le i$. 
Then
$$
\max \big\{F_0, \dots, F_i\big\}
\le 2(1+2\eta_1)\cdots (1+2\eta_{i} )
\bigg\{ (\frac12 +\eta_i)(F_0+F_1) +\eta_i\max \big\{ p_0, \dots, p_{i-1}\big\} \bigg\},
$$
where we have used the monotonicity of $\eta_j$.
This, together with (\ref{M-L-5}), gives
$$
\aligned
F_{i+1} & \le (1/2)^i F_1 +2\eta_{i+1} \max \big\{ p_0, \dots, p_i\big\}
+2\eta_{i+1} \max \big\{ F_0, \dots, F_i\big \}\\
&\le (1/2)^i F_1 +2\eta_{i+1} \max \big\{ p_0, \dots, p_i\big\}\\
&\qquad+2\eta_{i+1}
\cdot
2(1+2\eta_1)\cdots (1+2\eta_{i} )
 (\frac12 +\eta_i)(F_0+ F_1)\\
& \qquad +2\eta_{i+1}
\cdot
2(1+2\eta_1)\cdots (1+2\eta_{i} ) \cdot \eta_i\max \big\{ p_0, \dots, p_{i-1}\big\}\\
&\le 2(1+2\eta_1)\cdots (1+2\eta_{i+1})
\bigg\{ (2^{-i-1} +\eta_{i+1})(F_0 +F_1)
+\eta_{i+1} \max\big \{ p_0, \dots, p_i\big\} \bigg\} .
\endaligned
$$

Finally, we give the proof for estimate  (\ref{M-L-3}), which, together with (\ref{M-L-7}), yields (\ref{M-L-4}).
To this end we use (\ref{M-L-1}) and (\ref{M-L-7}) to obtain
$$
\aligned
p_{j+1}
&\le p_j + C_0 \max \{ F_j, F_{j+1} \}\\
& \le p_j +C \, (2^{-j-1} +\eta_{j+1} )(F_0 +F_1)
+C \, \eta_{j+1} \max \big\{ p_0, \dots, p_{j} \big\}\\
&\le (1+C\eta_{j+1} )\max \big\{ p_0,\dots, p_j\big\}
+C\,  (2^{-j-1} +\eta_{j+1}) (F_0 +F_1),
\endaligned
$$
where $C$ depends only on $C_0$ and $C_1$.
By a simple induction argument it follows
\begin{equation}
p_j \le  (1+C\eta_1)\cdots (1+C\eta_j)
\bigg\{ p_0 +F_0 +F_1 +C\, \sum_{k=1}^j (2^{-k} +\eta_k) (F_0 +F_1) \bigg\},
\end{equation}
where $C$ depends only on $C_0$ and $C_1$.
In view of (\ref{M-L-10}) this gives the desired estimate (\ref{M-L-3}).
The proof is now complete.
\end{proof}

\begin{theorem}\label{g-theorem-1}
Let $B_r=B(0,r)$ and
 $u\in L^2(B_1; \mathbb{R}^m)$. Let  $0<\varep<1/4$.
Suppose that for each $r\in (\varep, 1/4)$, there exists $w=w_r\in L^2(B_{r}; \mathbb{R}^m)$ such that
\begin{equation}\label{g-t-1-1}
\left\{ \average_{B_r} | u-w|^2\right\}^{1/2}
\le \eta (\varep/r) \left\{ \inf_{q\in \mathbb{R}^m}
\left(\average_{B_{2r}} |u-q |^2\right)^{1/2} + r\, K \right\},
\end{equation}
and
\begin{equation}\label{g-t-1-2}
\aligned
\frac{1}{\theta } 
& \inf_{\substack{ M\in \mathbb{R}^{m\times d}\\ 
q\in \mathbb{R}^m}} \left(\average_{B_{\theta r}} |w(x)-M x-q|^2\, dx \right)^{1/2}\\
&\qquad\qquad \le \frac{1}{2}
\inf_{\substack{ M\in \mathbb{R}^{m\times d}\\ 
q\in \mathbb{R}^m}}
  \left(\average_{B_{r}} |w(x)- M x-q|^2 \, dx \right)^{1/2},
 \endaligned
\end{equation}
where $K\ge 0$, $\theta\in (0,1/4)$, and $\eta (t)$ is a nondecreasing function on $(0,1]$. Assume that
\begin{equation}\label{g-t-1-3}
I=\int_0^1 \frac{\eta (t)}{t}\, dt<\infty.
\end{equation}
Then, for $\varep<t< (1/4)$,
\begin{equation}\label{g-t-1-4}
\frac{1}{t} \inf_{q\in \mathbb{R}^m}
\left\{ \average_{B_t} |u-q|^2\right\}^{1/2}
\le C \,  \left\{  K+\left(\average_{B_1} |u|^2\right)^{1/2} \right\},
\end{equation}
and
\begin{equation}\label{g-t-1-5}
\aligned
\frac{1}{t} \inf_{\substack{ M\in \mathbb{R}^{m\times d}\\ 
q\in \mathbb{R}^m}} &
\left\{ \average_{B_t} |u(x)-M x-q|^2 \, dx\right\}^{1/2}\\
&\qquad \le C \, \big\{  t^\alpha +\eta (\varep /t)\big\} 
\left\{  K+\left( \average_{B_1} |u|^2\right)^{1/2}\right\},
\endaligned
\end{equation}
where $\alpha=\alpha(\theta)>0$ and
$C$ depends only on $d$, $m$, $\theta$, and $I$.
\end{theorem}

\begin{proof}
It follows from the assumptions (\ref{g-t-1-2}) and (\ref{g-t-1-1}) that for $r\in (\varep, 1/2)$,
\begin{equation}\label{g-t-1-6}
\aligned
& \frac{1}{\theta r}
\inf_{\substack{ M\in \mathbb{R}^{m\times d}\\ 
q\in \mathbb{R}^m}}
 \left( \average_{B_{\theta r}} | u-Mx-q|^2\, dx\right)^{1/2} \\
& \le \frac{C}{r} \left\{\average_{B_r} |u-w_r|^2 \right\}^{1/2}
+\frac{1}{2r} 
\inf_{\substack{ M\in \mathbb{R}^{m\times d}\\ 
q\in \mathbb{R}^m}}
 \left(\average_{B_{r}} | u-M x-q|^2\, dx\right)^{1/2} \\
&\le C\, \eta(\varep/r)
\left\{ \frac{1}{2r} \inf_{q\in \mathbb{R}^m}
\left(\average_{B_{2r}} |u-q|^2 \right)^{1/2} +K \right\}\\
&\qquad\qquad\qquad
+\frac{1}{2r}  \inf_{\substack{ M\in \mathbb{R}^{m\times d}\\ 
q\in \mathbb{R}^m}}
 \left( \average_{B_{ r}} | u-Mx-q|^2\, dx\right)^{1/2}.
\endaligned
\end{equation}
Let $r_j =\theta^{j+1}$ for $0\le j\le \ell$, where $\ell$ is chosen so that $\theta^{\ell+2}<\varep\le \theta^{\ell+1}$
(we may assume that $\varep<\theta$).
Let
\begin{equation}\label{g-t-1-7}
\aligned
F_j  &=\frac{1}{r_j}
\inf_{\substack{ M\in \mathbb{R}^{m\times d}\\ 
q\in \mathbb{R}^m}}
\left( \average_{B_{r_j}} |u-Mx-q|^2\, dx \right)^{1/2}  \\
&=\frac{1}{r_j}
\inf_{q\in \mathbb{R}^m}
 \left(\average_{B_{r_j}} |u-M_j x -q|^2\, dx \right)^{1/2}
\endaligned
\end{equation}
and $p_j =|M_j|$.
Note that by (\ref{g-t-1-6}), 
\begin{equation}\label{g-t-8}
\aligned
F_{j+1} 
&\le \frac12 F_j + C \, \eta(\varep \theta^{-j-1})
\left\{ \frac{1}{2r_j} \inf_{q\in \mathbb{R}^m}
\left( \average_{B_{2r_j}} |u-q|^2\right)^{1/2} +K \right\}\\
& \le \frac12 F_j 
+ C\,  \eta(\varep \theta^{-j-1}) 
\big\{ K+F_{j-1} +p_{j-1} \big\}.
\endaligned
\end{equation}
Also observe that
$$
\aligned
& |M_{j+1} -M_j|
\le \frac{C}{r_{j+1}} \inf_{q\in \mathbb{R}^m}
\left\{ \average_{B_{r_{j+1}}} |(M_{j+1} -M_j) x-q |^2\, dx \right\}^{1/2}\\
&\le \frac{C}{r_{j+1}}
\inf_{q\in \mathbb{R}^m}
\left\{ \average_{B_{r_{j+1}}} |u- M_{j+1}x-q|^2 \right\}^{1/2}
+
\frac{C}{r_{j+1}}\inf_{q\in \mathbb{R}^m}
\left\{ \average_{B_{r_{j+1}}} |u- M_{j}x-q|^2 \right\}^{1/2}\\
& \le  C \, ( F_j +F_{j+1} ).
\endaligned
$$
This gives
\begin{equation}
p_{j+1} =|M_{j+1}|\le |M_j| + C\, (F_j + F_{j+1})
 = p_j +C\, (F_j + F_{j+1}).
\end{equation}
We further note that
$$
\sum_{j=0}^{\ell-1} \eta(\varep\theta^{-j-1})
\le \frac{1}{\ln (1/\theta)} \int_0^1 \frac{\eta (t)}{t}\, dt<\infty.
$$
Thus the sequences $\{ F_0, F_1, \dots, F_\ell \}$ and $\{ p_0, p_1, \dots, p_\ell\}$
satisfy the conditions in Lemma \ref{main-lemma-1}.
Consequently, we obtain
\begin{equation}
\aligned
F_j &\le C \, (2^{-j} +\eta(\varep \theta^{-j-1}) )
(K+ p_0 +F_0 +F_1)\\
&\le C \, (2^{-j} +\eta(\varep \theta^{-j-1}) )
\left\{  K + \left(\average_{B_1} |u|^2\right)^{1/2} \right\},
\endaligned
\end{equation}
and
\begin{equation}
\aligned
p_j  & \le C\, (  K + p_0 +F_0 +F_1)\\
&\le C 
\left\{  K + \left(\average_{B_1} |u|^2\right)^{1/2} \right\}.
\endaligned
\end{equation}

Finally, given any $t\in (\varep, \theta)$ (the case $t\ge \theta$ is trivial),
 we choose $j\ge 0$ so that $\theta^{j+2} < t\le \theta^{j+1} $.
 Then
 $$
 \aligned
 \frac{1}{t} 
 \inf_{\substack{ M\in \mathbb{R}^{m\times d}\\ 
q\in \mathbb{R}^m}} &
\left\{ \average_{B_t} |u-M x-q|^2 \right\}^{1/2} \le C\, F_j \\
&\le C \, \big\{ 2^{-j} +\eta(\varep \theta^{-j-1}) \big\}
\left\{  K + \left(\average_{B_1} |u|^2\right)^{1/2} \right\}\\
&\le C \big\{ t^\alpha +\eta(\varep/t)\big\}
\left\{ K + \left(\average_{B_1} |u|^2\right)^{1/2} \right\},
 \endaligned
 $$
 where $\alpha =\alpha (\theta)>0$,
and
 $$
 \aligned
 \frac{1}{t}
 \inf_{q\in \mathbb{R}^m}
 \left\{ \average_{B_t} |u-q|^2\right\}^{1/2}
 &\le \frac{C}{r_j} \inf_{q\in \mathbb{R}^m}
 \left\{ \average_{B_{r_j}} |u-q|^2\right\}^{1/2}\\
 & \le C\,\big\{ F_j +p_j\}\\
 & \le C \left\{  K + \left(\average_{B_1} |u|^2\right)^{1/2} \right\}.
 \endaligned
 $$
 This completes the proof.
\end{proof}

\begin{remark}
{\rm 
The $L^2$ norm plays no role in the proof above.
Theorem \ref{g-theorem-1} continues to hold if one replaces the $L^2$ average over $B_r$
by the $L^p$ average over $B_r$ for any $1\le p< \infty$ or by the $L^\infty$ norm over $B_r$.
}
\end{remark}

In the next section we will use Theorem \ref{g-theorem-1} 
to establish uniform interior Lipschitz estimates for $\mathcal{L}_\varep$.
The function $w=w_r(x)$  will be a suitably chosen solution of $\mathcal{L}_0 (w)=0$ in $B_r$.
Since the homogenized operator $\mathcal{L}_0$ has constant coefficients,
its solutions possess $C^{1, \alpha}$ estimates that make (\ref{g-t-1-2}) possible.
As we shall see in Sections 7 and 8, with our results on convergence rates in Section 2,
this approach for the interior  Lipschitz estimates may be adapted for boundary Lipschitz estimates
with either Dirichlet or Neumann conditions.



\section{Interior Lipschitz estimates}
\setcounter{equation}{0}

In this section we establish the uniform Lipschitz estimates
for $\mathcal{L}_\varep=-\text{\rm div } \big( A(x/\varep)\nabla \big)$.
Our approach is based on Theorem \ref{g-theorem-1}.
The key ingredients are provided by the next three lemmas.

\begin{lemma}\label{i-L-lemma-1}
Let $B_r=B(0, r)$.
Suppose that $u_\varep\in H^1 (B_{2r}; \mathbb{R}^m)$ and $\mathcal{L}_\varep (u_\varep) =0$ in $B_{2r}$
for some $0<\varep<r<1$. Then there exists $w\in H^1(B_r; \mathbb{R}^m)$ such that
$\mathcal{L}_0 (w)=0$ in $B_r$ and
\begin{equation}\label{i-L-1-00}
\left\{ \average_{B_r} |u_\varep -w|^2 \right\}^{1/2}
\le C_\delta\, \big[ \omega(\varep/r)\big]^{\frac23 -\delta}
\inf_{q\in \mathbb{R}^m}
\left\{ \average_{B_{2r}} |u_\varep -q|^2\right\}^{1/2}
\end{equation}
for any $\delta\in (0,1/4)$, where $\omega (t)=\omega_\sigma (t)$ is defined by (\ref{omega}). The constant 
$C_\delta$ depends only on $\delta$, $\sigma$, and $A$.
\end{lemma}

\begin{proof}
By a simple rescaling we may assume that $r=1$.
By subtracting a constant we may also assume $ \int_{B_2} u_\varep=0$.
Let $f_t =u_\varep * \varphi_t$, where $\varphi_t (x)=t^{-d} \varphi (x/t)$,
$\varphi\in C_0^\infty(B_1)$ and
$\int_{\mathbb{R}^d} \varphi =1$. Since $\mathcal{L}_\varep (u_\varep)=0$ in $B_2$,
using the interior H\"older estimate for $\mathcal{L}_\varep$,
$$
\| u_\varep\|_{C^\beta (B_{7/4})}\le C_\beta\,  \| u_\varep \|_{L^2(B_2)} \qquad \text{ for any } 0<\beta<1
$$
(see Theorem 3.4 in \cite{Shen-2014}),
it is easy to see that 
\begin{equation}\label{i-L-1-10}
\| f_t- u_\varep \|_{C^{ \alpha} (B_{3/2})}\le C_{\alpha, \beta} \, t^{\beta-\alpha} \,\| u_\varep\|_{L^2(B_2)}
\end{equation}
and  
\begin{equation}\label{i-L-1-11}
\|  f_t \|_{C^{1, \alpha} (B_{3/2})}
\le C_{\alpha, \beta} \, t^{\beta-\alpha-1} \, \| u_\varep \|_{L^2(B_2)},
\end{equation}
where $t\in (0,1/4)$ and $0<\alpha<\beta<1$.
We now solve the Dirichlet problems
\begin{equation}
\left\{
\aligned
\mathcal{L}_\varep (v_\varep) & =0 & \quad & \text{ in } B_{5/4},\\
v_\varep &=f_t &\quad &\text{ on } \partial B_{5/4},
\endaligned
\right.
\quad 
\text{ and }
\quad
\left\{
\aligned
\mathcal{L}_0 (w) & =0 & \quad & \text{ in } B_{5/4},\\
w &=f_t &\quad &\text{ on } \partial B_{5/4},
\endaligned
\right.
\end{equation}
where $t\in (0,1/4)$ is to be determined.
Since $\mathcal{L}_\varep (u_\varep -v_\varep)=0$
 in $B_{5/4}$, it follows from (\ref{H-estimate}) and (\ref{i-L-1-10}) that
\begin{equation}\label{i-L-1-1}
\| u_\varep -v_\varep\|_{L^\infty(B_{5/4})}
\le C_\alpha\, \| u_\varep -f_t\|_{C^\alpha(\partial B_{5/4})}
\le C_{\alpha, \beta} \, t^{\beta-\alpha} \, \| u_\varep\|_{L^2(B_2)}.
\end{equation}
Also, observe that by Theorem \ref{rate-theorem-1},
\begin{equation}\label{i-L-1-2}
\aligned
\| v_\varep-w\|_{L^2(B_{5/4})}
& \le C_\alpha \, \big[\omega (\varep )\big]^{2/3} \, \| f_t \|_{C^{1, \alpha}(\partial B_{5/4})}\\
& \le C_{\alpha, \beta}\, t^{\beta-\alpha-1} \, \big[\omega(\varep)\big]^{2/3}\, \| u_\varep\|_{L^2(B_2)}.
\endaligned
\end{equation}
In view of (\ref{i-L-1-1}) and (\ref{i-L-1-2}) we obtain 
\begin{equation}
\aligned
\| u_\varep -w\|_{L^2(B_1)}
&\le \| u_\varep -v_\varep\|_{L^2(B_1)} +\| v_\varep-w\|_{L^2(B_1)}\\
&\le C_{\alpha, \beta}\, t^{\beta-\alpha}
\big\{ 1+ t^{-1} \big[\omega (\varep)\big]^{2/3}  \big\} \| u_\varep \|_{L^2(B_2)}.
\endaligned
\end{equation}
We now choose $t= c\, \big[\eta (\varep)\big]^{2/3} \in (0, 1/4)$, $\alpha=(3/4)\delta$, and
$\beta=1-\alpha$, where $\delta \in (0,1/4)$. This gives
\begin{equation}
\aligned
\| u_\varep -w\|_{L^2(B_1)}
&\le C_{\delta} \, \big[ \omega(\varep) \big]^{\frac23-\delta} \| u_\varep\|_{L^2(B_2)}\\
&\le C_{\delta} \, \big[ \omega(\varep) \big]^{\frac23-\delta}
\inf_{q\in \mathbb{R}^m}
\left\{ \average_{B_2} |u_\varep -q|^2\right\}^{1/2},
\endaligned
\end{equation}
where we have used the fact $\int_{B_2} u_\varep =0$ for the last inequality.
\end{proof}

\begin{lemma}\label{i-L-lemma-0}
Suppose $w\in H^1(B_r; \mathbb{R}^m)$ and $\mathcal{L}_0 (w)=0$ in $B_r$, where $B_r=B(0,r)$.
Then, for any $\theta\in (0,1/2)$,
\begin{equation}\label{g-t-interior-1}
\frac{1}{\theta } \inf_{\substack{ M\in \mathbb{R}^{m\times d} \\ q\in \mathbb{R}^m}}
\left\{ \average_{B_{\theta r}} |w-Mx -q|^2\right\}^{1/2}
\le {C\theta}{ } \inf_{\substack{ M\in \mathbb{R}^{m\times d} \\ q\in \mathbb{R}^m}}
\left\{ \average_{B_{ r}} |w-Mx -q|^2\right\}^{1/2},
\end{equation}
where $C$ depends only on $d$, $m$, and $\mu$.
As a result, by choosing $\theta$ so small that $C\theta<(1/2)$,
solutions of $\mathcal{L}_0 (w)=0$ in $B_r$
satisfy the condition (\ref{g-t-1-2}) 
in Theorem \ref{g-theorem-1}.
\end{lemma}

\begin{proof}
Estimate (\ref{g-t-interior-1}) follows readily from the interior $C^2$ estimates for $\mathcal{L}_0$.
Indeed, by rescaling, we may assume that $r=1$.
In this case the left hand side of (\ref{g-t-interior-1}) is bounded by $C\, \theta \|\nabla^2 w\|_{L^\infty(B_\theta)}$.
Since $\mathcal{L}_0 (w-Mx -q)=0$ in $B_1$ for any $M\in \mathbb{R}^{m\times d}$ and
$q\in \mathbb{R}^m$, by the $C^2$ estimates for $\mathcal{L}_0$,
\begin{equation}
\theta \| \nabla^2 w\|_{L^\infty(B_\theta)}
\le \theta
\|\nabla^2 w\|_{L^\infty(B_{1/2})}
\le C \theta \inf_{\substack{ M\in \mathbb{R}^{m\times d} \\ q\in \mathbb{R}^m}}
\left\{ \average_{B_1} |w-Mx -q|^2\right\}^{1/2},
\end{equation}
where $C$ depends only on $d$, $m$, and $\mu$.
The proof is complete.
\end{proof}

\begin{lemma}\label{Dini-lemma}
Suppose that there exist $C_0>0$ and $N>(5/2)$ such that
$$
\rho(R)\le C_0 \, \big[ \log R \big]^{-N} \qquad \text{ for any } R\ge 2,
$$
where $\rho(R)$ is defined by (\ref{rho}).
Then there exist $\sigma \in (0,1)$ and $\delta\in (0,1/4)$ such that
$$
\int_0^1 \big[ \omega_\sigma (t) \big]^{\frac{2}{3} -\delta} \, \frac{dt}{t}<\infty,
$$
where $\omega_\sigma (t)$ is defined by (\ref{omega}).
\end{lemma}

\begin{proof}
It follows from the definition of $\Theta_\sigma (T)$ that
$$
\Theta_\sigma (T) \le \rho(\sqrt{T}) +\left(\frac{1}{\sqrt{T}}\right)^\sigma
\le C_\sigma \, \big[ \log T\big]^{-N}
$$
for $T\ge 2$.
Also, it was proved in \cite[Theorem 6.6]{Shen-2014} that
$$
\langle |\psi-\nabla \chi_T|\rangle \le C_\sigma \int_{T/2}^\infty \frac{\Theta_\sigma (r)}{r}\, dr
$$
for any $\sigma \in (0,1)$.
As a result, if $\sigma=1-N^{-1}$, we obtain
$$
\aligned
\eta_\sigma (t) & =\big[ \Theta_1 (t^{-1}) \big]^{\sigma} +\sup_{T\ge t^{-1}} \langle |\psi -\nabla \chi_T |\rangle\\
& \le \big[ \Theta_1 (t^{-1}) \big]^{\sigma} 
+C_\sigma \int_{(2t)^{-1}}^\infty \frac{\Theta_\sigma (r)}{r}\, dr\\
&\le C_\sigma \big[ \log (1/t) \big]^{1-N}
\endaligned
$$
for $t\in (0, 1/2)$.
Finally, since $N>(5/2)$, we may choose $\delta\in (0,1/4)$ so small that
$((2/3)-\delta)(1-N)<-1$.
This leads to
$$
\int_0^1 \big[ \eta_\sigma (t) \big]^{\frac{2}{3} -\delta} \, \frac{dt}{t}
\le C + C \int_0^{1/2} \big[ \log (1/t) \big]^{(1-N)(\frac23-\delta)} \, \frac{dt}{t}
<\infty,
$$
and completes the proof.
\end{proof}

We are now ready to prove the interior Lipschitz estimates for $\mathcal{L}_\varep$.
We first treat the case $\mathcal{L}_\varep (u_\varep)=0$.

\begin{lemma}\label{i-L-lemma-3}
Suppose that $A(y)$ satisfies the same conditions as in Theorem \ref{main-theorem-Lip}.
Let $u_\varep\in H^1(2B; \mathbb{R}^m)$ be a weak solution of $\mathcal{L}_\varep (u_\varep)=0$ in 
$2B$, where $B=B(x_0, r)$ for some $x_0\in \mathbb{R}^d$ and $r>0$.
Then $|\nabla u_\varep| \in L^\infty(B)$ and
\begin{equation}\label{interior-Lip}
\| \nabla u_\varep\|_{L^\infty(B)}
\le \frac{C}{r} \left\{ \average_{2B} | u_\varep|^2\right\}^{1/2},
\end{equation}
where $C$ depends only on $A$.
\end{lemma}

\begin{proof}
By translation and dilation it suffices to prove that
\begin{equation}\label{i-L-1}
|\nabla u_\varep (0)|\le C \left\{ \average_{B(0,1)} |u_\varep|^2\right\}^{1/2},
\end{equation}
if $\mathcal{L}_\varep (u_\varep)=0$ in $B(0,1)$.
Note that we only need to treat the case $0<\varep<(1/4)$, since
the case $\varep\ge (1/4)$ follows from the standard local regularity theory for second-order elliptic systems 
with H\"older continuous coefficients.
Let $v_\varep (x)=\varep^{-1} u_\varep (\varep x)$.
Then $\mathcal{L}_1 (v_\varep)=0$ in $B(0, 2\varep^{-1})$.
By the standard regularity theory for $\mathcal{L}_1$,
$$
\aligned
|\nabla u_\varep (0)| &=|\nabla v_\varep (0)|
\le C \inf_{q\in \mathbb{R}^m} \left\{ \average_{B(0,2)} |v_\varep -q|^2\right\}^{1/2}\\
&= C\inf_{q\in \mathbb{R}^m}\frac{1}{\varep}
\left\{ \average_{B(0, 2\varep)} |u_\varep -q|^2 \right\}^{1/2}.
\endaligned
$$

To complete the proof
we  use Theorem \ref{g-theorem-1}, with $K=0$,  to obtain
\begin{equation}\label{i-L-3}
\frac{1}{\varep}\inf_{q\in \mathbb{R}^m}
\left\{ \average_{B(0, 2\varep)} |u_\varep -q|^2 \right\}^{1/2}
\le 
C \left\{ \average_{B(0,1)} |u_\varep|^2\right\}^{1/2}.
\end{equation}
Note that the condition (\ref{g-t-1-1})  is given by Lemma \ref{i-L-lemma-1},
while the condition (\ref{g-t-1-2}) is given by Lemma \ref{i-L-lemma-0}.
 Also,  the Dini condition (\ref{g-t-1-3}) is satisfied in view of Lemma \ref{Dini-lemma}.
As a result, the estimate (\ref{i-L-3}) follows from (\ref{g-t-1-4})
with $t=2\varep$.
\end{proof}

\begin{remark}
\label{Liouville}
{\rm 
In the argument for Lemma 4.4, we used only the first conclusion~\eqref{g-t-1-4} of Theorem~\ref{g-theorem-1}. The second conclusion (\ref{g-t-1-5}) is also useful, and yields the following Liouville result: if $u\in H^1(\mathbb{R}^d;\mathbb{R}^m)$ is any solution of $\mathcal L_1(u) = 0$ in $\mathbb{R}^d$ satisfying the linear growth condition
\begin{equation*} \label{}
\limsup_{r\to\infty} \frac 1r \left\{ \average_{B(0,r)} |u|^2 \right\}^{1/2} <\infty, 
\end{equation*}
then there exists $M\in \mathbb{R}^{m\times d}$ such that 
\begin{equation*} \label{}
\limsup_{r\to \infty} \frac 1r \left\{ \average_{B(0,r)} |u(x) - Mx |^2\,dx \right\}^{1/2} = 0.
\end{equation*}
In other words, if an entire solutions grows at most linearly, it is close to an affine function. 
To prove this, we follow the argument of Lemma 4.4 with $\varep>0$ fixed and $u_\varep(x):= \varep u(x/\varep)$. 
We notice that in the  application of Theorem~\ref{g-theorem-1} we invoked to get~\eqref{i-L-3}, 
we also obtain from the second conclusion of the theorem that, for every $\varep < t < 1/4$, 
\begin{equation*} \label{}
\frac{1}{t}\inf_{\substack{ M\in \mathbb{R}^{m\times d}\\ 
q\in \mathbb{R}^m}}
\left\{ \average_{B(0, 2 t)} |u_\varep(x)-Mx -q|^2\,dx \right\}^{1/2}
\le 
C \left\{ t^\alpha + \eta(\varep/t) \right\} \left\{ \average_{B(0,1)} |u_\varep|^2\right\}^{1/2}.
\end{equation*}
By undoing the scaling and writing this in terms of $u$, we obtain, for every $1 < r < 1/4\varep$,
\begin{equation*} \label{}
\frac{1}{r}\inf_{\substack{ M\in \mathbb{R}^{m\times d}\\ 
q\in \mathbb{R}^m}}
\left\{ \average_{B(0, 2 r)} |u(x)-Mx -q|^2\,dx \right\}^{1/2}
\le 
C \left\{ (\varep r)^\alpha + \eta(1/r) \right\} \varep \left\{ \average_{B(0,1/\varep)} |u|^2\right\}^{1/2}.
\end{equation*}
Sending $\varep\to 0$ and using the growth hypothesis, we get
\begin{equation*} \label{}
\frac{1}{r}\inf_{\substack{ M\in \mathbb{R}^{m\times d}\\ 
q\in \mathbb{R}^m}}
\left\{ \average_{B(0, 2 r)} |u(x)-Mx -q|^2\,dx \right\}^{1/2}
\le 
C \eta(1/r).
\end{equation*}
We now obtain the Liouville property by applying the previous inequality  on the dyadic scales $r_k:=2^k$, $k\in\mathbb{N}$, and using the Dini condition~\eqref{g-t-1-3} to verify that the sequence $\{M_k\}_{k\in\mathbb{N}}\subset \mathbb{R}^m$ of corresponding affine approximations is a Cauchy sequence. 
}
\end{remark}

As we were finishing the writing of this paper, we became aware of some very recent results of Gloria, Neukamm, and Otto~\cite{GNO-2014}, who obtain a more general version of the Liouville result presented above in Remark~\ref{Liouville}. Their scheme is similar to the one from~\cite{Armstrong-Smart-2014}, which we use here. 
Both are based on a Campanato iteration to obtain an improvement of flatness for solutions,
 although ``flatness" in~\cite{Armstrong-Smart-2014}, and in this paper,
  is defined with respect to \emph{affine functions}, while~\cite{GNO-2014}, following~\cite{AL-1987, AL-liouv}, 
  defines it with respect to \emph{correctors}. The latter notion allows to formulate some more precise results, although it does not seem to help estimating the gradient of the correctors themselves (which is more or less 
  equivalent to the task of obtaining uniform Lipschitz estimates). 

\smallskip

\begin{theorem}\label{i-L-theorem}
Suppose that $A(y)$ satisfies the same conditions as in Theorem \ref{main-theorem-Lip}.
Let $u_\varep$ be a weak solution of
 $\mathcal{L}_\varep (u_\varep)=F$ in $2B$, where $B=B(x_0, r)$. Then
\begin{equation}\label{estimate-i-L}
\|\nabla u_\varep\|_{L^\infty(B)}
\le \frac{C}{r} \left\{ \average_{2B} |u_\varep|^2 \right\}^{1/2}
+C \, r^\beta\sup_{\substack{y\in 2B\\ 0<t<r}} t^{1-\beta}  \average_{B(y, t)\cap 2B} |F|
\end{equation}
for any $\beta\in (0,1)$, where $C$ depends only on $\beta$ and $A$.
\end{theorem}

\begin{proof}
By translation and dilation we may assume that $x_0=0$ and $r=1$.
We may also assume $d\ge 3$, as the 2-d case may be reduced to the 3-d case by adding a dummy variable.

Consider
$$
v_\varep (x)=\int_{2B} \Gamma_\varep (x,y) F(y)\, dy,
$$
where $\Gamma_\varep (x,y)$ denotes the matrix of fundamental solutions for $\mathcal{L}_\varep$
in $\mathbb{R}^d$, with pole at $y$. Note that $\mathcal{L}_\varep (v_\varep)=F$ in $2B$.
By the interior H\"older estimates in \cite{Shen-2014}, we have
\begin{equation}\label{size-estimate}
|\Gamma_\varep (x,y)|\le C\, |x-y|^{2-d} \qquad \text{ for any } x, y\in \mathbb{R}^d,
\end{equation}
where $C$ depends only on $A$.
Since $\mathcal{L}_\varep \big (\Gamma_\varep (\cdot, y) \big)=0$ in $\mathbb{R}^d\setminus \{ y\}$,
we may use (\ref{size-estimate}) and (\ref{interior-Lip}) to obtain 
\begin{equation}\label{gradient-estimate}
|\nabla_x \Gamma_\varep (x,y)|\le C\, |x-y|^{1-d} \qquad \text{ for  any } x,y\in \mathbb{R}^d.
\end{equation}
It is not hard to see that this gives
\begin{equation}\label{estimate-v}
\|\nabla v_\varep\|_{L^\infty(2B)} +\| v_\varep\|_{L^\infty(2B)}
\le C_\beta\, \sup_{\substack{y\in 2B\\ 0<t<1}} t^{1-\beta} \average_{ B(y,t)\cap 2B} |F|.
\end{equation}
for any $\beta\in (0,1)$.

Finally, since $\mathcal{L}_\varep (u_\varep -v_\varep)=0$ in $2B$,
we may invoke Lemma \ref{i-L-lemma-3} to obtain 
$$
\aligned
\|\nabla (u_\varep -v_\varep)\|_{L^\infty(B)}
&\le C \left\{ \average_{2B} |u_\varep -v_\varep|^2\right\}^{1/2}\\
&\le C \left\{ \average_{2B}|u_\varep|^2\right\}^{1/2}
+C_\beta \sup_{\substack{y\in 2B\\ 0<t<1}} t^{1-\beta} \average_{B(y,t)\cap 2B} |F|,
\endaligned
$$
where we have used (\ref{estimate-v}) for the last inequality.
This, together with (\ref{estimate-v}), yields the estimate (\ref{estimate-i-L}).
\end{proof}

\begin{remark}\label{ac-remark}
{\rm
Fix $1\le j\le d$, $1\le \beta\le m$, and $y\in \mathbb{R}^d$.
Let 
$$
u(x)= \chi_{T, j}^\beta (x) -\chi_{T, j}^\beta (y) + (x_j-y_j) e^\beta,
$$
where $T\ge 1$.
Then $\mathcal{L}_1 (u) = -T^{-2} \chi_{T, j}^\beta$ in $\mathbb{R}^d$.
It follows from Theorem \ref{i-L-theorem} that
$$
|\nabla u (y)|\le \frac{C}{T}  \left(\average_{B(y, T)} |u|^2 \right)^{1/2} + C\, T^{-1} \| \chi_T\|_\infty
\le C.
$$
As a result, if $A$ satisfies the same conditions in Theorem \ref{main-theorem-Lip}, then
\begin{equation}\label{ac-Lip-bound}
\|\nabla \chi_T \|_{L^\infty(\mathbb{R}^d)} \le C,
\end{equation}
where $C$ depends only on $A$.
}
\end{remark}



\section{Interior $W^{1,p}$ estimates}
\setcounter{equation}{0}

The goal of this section is to prove the following theorem.

\begin{theorem}\label{i-W-1-p-theorem}
Suppose that $A(y)$ is uniformly almost-periodic in $\mathbb{R}^d$ and
satisfies (\ref{ellipticity}).
Also assume  $A(y)$ satisfies the  condition (\ref{decay-condition}) for some $N>5/2$. 
Let $u_\varep \in H^1(2B;\mathbb{R}^m)$ be a weak solution of
$\mathcal{L}_\varep (u_\varep) =\text{\rm div} (f)$ in $2B$ for some ball $B$ in $\mathbb{R}^d$.
 Suppose that $f=(f_i^\alpha)\in L^p (2B; \mathbb{R}^{dm})$ for some $2<p<\infty$.
Then
\begin{equation}\label{i-W-1-p}
\left\{ \average_B |\nabla u_\varep|^p \right\}^{1/p}
\le C_p \left\{ \left(\average_{2B} |\nabla u_\varep|^2 \right)^{1/2}
+\left(\average_{2B} |f|^p \right)^{1/p} \right\},
\end{equation}
where $C_p$ depends only on $p$ and $A$.
\end{theorem}

We remark that in contrast to Theorem \ref{i-L-theorem}, the H\"older continuity 
condition (\ref{H-continuity}) is not required for $W^{1,p}$ estimates.

We first treat the case where $\mathcal{L}_\varep (u_\varep) =0$.

\begin{lemma}\label{interior-W-1-p-lemma}
Assume $A$ satisfies the same assumptions as in Theorem \ref{i-W-1-p-theorem}.
Let $u_\varep\in H^1(2B; \mathbb{R}^m)$ be a weak solution of 
$\mathcal{L}_\varep (u_\varep)=0$ in $2B$, where $B=B(x_0, r)$
for some $x_0\in \mathbb{R}^d$ and $r>0$.
Then $|\nabla u_\varep |\in L^p(B)$ for any $2<p<\infty$, and 
\begin{equation}\label{interior-W-1-p}
\left\{\average_{B} |\nabla u_\varep|^p \right\}^{1/p}
\le C_p
\left\{ \average_{2B} |\nabla u_\varep|^2\right\}^{1/2},
\end{equation}
where $C_p$ depends only on $p$ and $A$.
\end{lemma}

\begin{proof} Fix $2<p<\infty$.
By translation and dilation we may assume $x_0=0$ and $r=1$.
By subtracting a constant we may also assume $\int_{2B} u_\varep =0$.
We may further assume that $0<\varep<(1/4)$, as the case $\varep\ge (1/4)$
follows from the standard local $W^{1,p}$ estimates
for second-order elliptic systems with continuous coefficients.
By rescaling the same theory  also gives
\begin{equation}\label{i-W-1-p-1}
\left\{ \average_{B(0, \varep)} |\nabla u_\varep|^p\right\}^{1/p}
\le \frac{C_p}{\varep} 
\inf_{q\in \mathbb{R}^m}
\left\{ \average_{B(0, 2\varep)} |u_\varep -q|^2 \right\}^{1/2}.
\end{equation}
An inspection of the proof for Lemma  \ref{i-L-lemma-3} shows that
estimate (\ref{i-L-3}) continues to hold under the assumption in Theorem \ref{i-W-1-p-theorem}
(the H\"older continuity of $A$ is not required).
Thus, 
$$
\left\{ \average_{B(0, \varep)} |\nabla u_\varep|^p\right\}^{1/p}
\le 
C_p \, \| u_\varep\|_{L^2(B(0,1))}.
$$
By translation this implies that for any $z\in B(0,1)$,
$$
\int_{B(z,\varep)} |\nabla u_\varep|^p\, dx
\le C_p \, \varep^d\, \| u_\varep\|^p_{L^2(B(0,2))}.
$$
It follows  by a simple covering argument that
$$
\int_{B(0,1)} |\nabla u_\varep|^p\le C_p\,  \| u_\varep\|^p_{L^2(B(0,2))}
\le C_p\,  \|\nabla u_\varep\|_{L^2(B(0,2))}^p,
$$
where we have used Poincar\'e inequality for the last step.
\end{proof}

The reduction of Theorem \ref{i-W-1-p-theorem} to Lemma \ref{interior-W-1-p-lemma}
is done through a refined version of Calder\'on-Zygmund argument due to Caffarelli and Peral in \cite{CP-1998}.
Motivated by \cite{CP-1998}, the following theorem
 was formulated and proved in \cite{Shen-2007} (also see \cite{Shen-2005}).
 
 \begin{theorem}\label{Shen-theorem}
 Let $F\in L^2(4B_0)$ and $f\in L^p(4B_0)$ for some $2<p<q<\infty$, where $B_0$ is a ball in $\mathbb{R}^d$.
 Suppose that for each ball $B\subset 2B_0$ with $|B|\le c_1 |B_0|$, there exist two measurable functions
 $F_B$ and $R_B$ on $2B$, such that 
 $|F|\le |F_B| +|R_B|$ on $2B$, and
 \begin{equation}\label{Shen-1}
 \aligned
 \left\{ \average_{2B} |R_B|^q \right\}^{1/q}
 &\le C_1 \left\{ \left(\average_{c_2 B} |F|^2 \right)^{1/2}
 +\sup_{4B_0\supset B^\prime\supset B} \left(\average_{B^\prime} |f|^2\right)^{1/2} \right\},\\
 \left\{ \average_{2B} |F_B|^2 \right\}^{1/2}
 & \le C_2 \sup_{4B_0\supset B^\prime\supset B} \left\{ \average_{B^\prime} |f|^2 \right\}^{1/2},
 \endaligned
 \end{equation}
 where $C_1, C_2>0$, $0<c_1<1$, and $c_2> 2$. Then $F\in L^p (B_0)$ and
 \begin{equation}
 \left\{\average_{B_0} |F|^p\right\}^{1/p}
 \le C \left\{ \left(\average_{4B_0} |F|^2 \right)^{1/2} 
 +\left(\average_{4B_0} |f|^p \right)^{1/p} \right\},
 \end{equation}
 where $C$ depends only on $d$,  $C_1$, $C_2$, $c_1$, $c_2$, $p$, and $q$.
 \end{theorem}

\begin{proof}[\bf Proof of Theorem \ref{i-W-1-p-theorem}]
Suppose that $\mathcal{L}_\varep (u_\varep)=\text{\rm div} (f)$ in $2B_0$
and $f\in L^p(2B_0; \mathbb{R}^{dm})$ for some $2<p<\infty$.
Let $q=p+1$.
We will apply Theorem \ref{Shen-theorem} to $F=|\nabla u_\varep|$.
For each ball $B$ such that $4B\subset 2B_0$, we write $u_\varep=v_\varep +w_\varep$ on $2B$, where
$v_\varep \in H^1_0(4B; \mathbb{R}^{dm})$ is the solution to
$\mathcal{L}_\varep (v_\varep)=\text{\rm div} (f)$ in $4B$.
Let
$$
F_B =|\nabla v_\varep| \qquad \text{ and } \qquad R_B =|\nabla w_\varep|.
$$
Then $|F|\le F_B + R_B$ on $2B$.
It is easy to see that the first inequality in (\ref{Shen-1}) follows from the energy estimate.
Since $\mathcal{L}_\varep (w_\varep)=0$ in $4B$,
it follows from Lemma \ref{interior-W-1-p-lemma} that
$$
\aligned
\left\{\average_{2B} |R_B|^q\right\}^{1/q}
&\le C \left\{ \average_{4B} |R_B|^2\right\}^{1/2}\\
&\le C \left\{ \average_{4B} |\nabla u_\varep|^2\right\}^{1/2}
+\left\{ \average_{4B} |\nabla v_\varep|^2\right\}^{1/2}\\
&\le C\left\{ \average_{4B} |F|^2\right\}^{1/2}
+C \left\{ \average_{4B} |f|^2\right\}^{1/2},
\endaligned
$$
where we have used the energy estimate for the last inequality.
This give the second inequality in (\ref{Shen-1}).
It then follows by Theorem \ref{Shen-theorem} that
$$
\left\{ \average_{B} |\nabla u_\varep|^p\right\}^{1/p}
\le C \left\{ \left(\average_{4B} |\nabla u_\varep|^2\right)^{1/2}
+\left(\average_{4B} |f|^p\right)^{1/p} \right\}
$$
for any ball $B$ such that $4B\subset 2B_0$.
By a simple covering argument this gives (\ref{i-W-1-p}) for $B=B_0$.
\end{proof}



\section{Boundary $W^{1, p}$ estimates
and proof of Theorems \ref{main-theorem-W-1-p} and \ref{main-theorem-W-1-p-N}}
\setcounter{equation}{0}

In this section we establish uniform boundary $W^{1,p}$ estimate  for 
$\mathcal{L}_\varep$ with Dirichlet or Neumann condition.
As we shall see, boundary $W^{1,p}$ estimates follow from the interior $W^{1,p}$ estimates
and boundary H\"older estimates.

For $r>0$, let
\begin{equation}\label{Delta}
\aligned
D_r  & =\big\{ (x^\prime, x_d)\in \mathbb{R}^d: \,
 |x^\prime|<r \text{ and } \phi(x^\prime)<x_d<\phi(x^\prime) + 10(K_0+1)r \big\},\\
\Delta_r &= \big\{ (x^\prime, \phi(x^\prime))\in \mathbb{R}^d:\, |x^\prime|<r \big\},
\endaligned
\end{equation}
where $\phi : \mathbb{R}^{d-1} \to \mathbb{R}$ is a $C^{1,\alpha}$ function such that 
supp$(\phi)\subset \{ x^\prime\in \mathbb{R}^{d-1}: \ |x^\prime|\le 10 \}$,
\begin{equation}\label{phi}
\phi (0)=0, \quad \nabla \phi (0)=0, \quad \text{ and } \quad \|\nabla \phi\|_{C^{\alpha} (\mathbb{R}^{d-1})}\le K_0+1.
\end{equation}
The constant $K_0>0$ in (\ref{phi}) is fixed.
The bounding constants $C$ in the next two lemmas will depend on $(\alpha, K_0)$,
 but otherwise not directly on $\phi$.

\begin{lemma}\label{b-H-lemma}
Suppose that $A$ is uniformly almost-periodic in $\mathbb{R}^d$ and satisfies the ellipticity
condition (\ref{ellipticity}).
Let $u_\varep\in H^1(D_{2r}; \mathbb{R}^m)$
be a weak solution of $\mathcal{L}_\varep (u_\varep)=0$ in $D_{2r}$,
with either $u_\varep=0$ or $\frac{\partial u_\varep}{\partial \nu_\varep}=0$ on $\Delta_{2r}$,
for some $0<r\le 1$.
Then, for any $0<\beta<1$,
\begin{equation}\label{b-H-00}
|u_\varep (x) - u_\varep (y)|
\le C_\beta\, r\, \left(\frac{|x-y|}{r}\right)^\beta \left(\average_{D_{2r}} |\nabla u_\varep|^2\right)^{1/2},
\end{equation}
where $C_\beta$ depends only on $\beta$, $A$, and $(\alpha, K_0)$ in (\ref{phi}).
\end{lemma}

\begin{proof}
In the case of Dirichlet condition $u_\varep =0$ on $\Delta_{2r}$,
the estimate (\ref{b-H-00}) was proved in \cite{Shen-2014}
by using  a three-step compactness argument introduced in \cite{AL-1987}.
The compactness argument in \cite{AL-1987} for H\"older
estimates does not involve correctors and extends readily to the almost-periodic setting.
This is also true in the case of Neumann boundary conditions.
We omit the details and refer the reader to \cite{KLS1}, where  uniform 
boundary H\"older estimates with Neumann conditions were established in the periodic setting.
\end{proof}

\begin{lemma}\label{b-W-lemma-1}
Suppose that $A$ is uniformly almost-periodic in $\mathbb{R}^d$ and satisfies (\ref{ellipticity}).
Also assume that the decay condition (\ref{decay-condition}) holds for some $C_0>0$ and $N>(3/2)$.
Let $u_\varep\in H^1(D_{2r}; \mathbb{R}^m)$
be a weak solution of $\mathcal{L}_\varep (u_\varep)=0$ in $D_{2r}$,
with either $u_\varep =0$  or $\frac{\partial u_\varep}{\partial \nu_\varep}
=0$ on $\Delta_{2r}$, for some $0<r\le1$.
Then, for any $2<p<\infty$,
\begin{equation}\label{b-W-1-00}
\left\{ \average_{D_r} |\nabla u_\varep|^p\right\}^{1/p}
\le C_p 
\left\{ \average_{D_{2r}} |\nabla u_\varep|^2\right\}^{1/2},
\end{equation}
where $C_p$ depends only on $p$, $A$, and $(\alpha, K_0)$ in (\ref{phi}).
\end{lemma}

\begin{proof}
By rescaling we may assume $r=1$. Also assume that $\| \nabla u_\varep\|_{L^2(D_2)} =1$.
Let $\delta (x)=\text{\rm dist} (x, \partial D_2)$.
It follows from the interior $W^{1,p}$ estimates in Theorem \ref{i-W-1-p-theorem}
that 
\begin{equation}\label{b-W-001}
\aligned
\left(\average_{B(y, \delta(y)/8)} |\nabla u_\varep|^p \right)^{1/p}
 & \le C \left(\average_{B(y, \delta(y)/4)} |\nabla u_\varep|^2 \right)^{1/2}\\
&\le C \left(\average_{B(y, \delta(y)/2)} |u_\varep (x)-u_\varep (y)|^2\, dx\right)^{1/2}\\
&\le C_\beta  \big[ \delta (y) \big]^{\beta -1}
\endaligned
\end{equation}
for any $\beta \in (0,1)$, where we have used Lemma \ref{b-H-lemma}
for the last inequality. By choosing $\beta \in (1-\frac{1}{p}, 1)$, this implies that
\begin{equation}\label{b-W-002}
\int_{D_1} \left(\average_{B(y, \delta(y)/8)} |\nabla u_\varep (x)|^p \, dx \right)\, dy \le C.
\end{equation}
By Fubini's Theorem we then obtain 
\begin{equation}\label{b-W-003}
\int_{D_1} |\nabla u(x)|^p \left\{ \int_{\{ y\in D_1: \, |y-x|<\frac{\delta(y)}{8} \} }
\frac{dy}{\big[\delta (y) \big]^d} \right\}dx \le C.
\end{equation}

Finally, we note that if $|y-x|<\frac{\delta(y)}{8}$, then $\delta (x)\approx \delta (y)$.
Also, it is not hard to verify that for $x\in D_1$,
$$
D_1 \cap B(x, \delta(x)/16)\subset \{ y\in D_1: \, |y-x|<\delta(y)/8 \}.
$$
It follows that
$$
\int_{\{ y\in D_1: \, |y-x|<\frac{\delta(y)}{8} \} }
\frac{dy}{\big[\delta (y) \big]^d} \ge c>0.
$$
This, together with (\ref{b-W-003}), gives
$$
\int_{D_1} |\nabla u_\varep (x)|^p\, dx\le C,
$$
and completes the proof.
\end{proof}

\begin{theorem}\label{b-W-1-p-theorem}
Suppose that $A$ is uniformly almost-periodic in $\mathbb{R}^d$ and satisfies (\ref{ellipticity}).
Also assume that the decay condition (\ref{decay-condition}) holds for some $C_0>0$ and $N>(3/2)$.
Let $\Omega$ be a bounded $C^{1, \alpha}$ domain in $\mathbb{R}^d$ for some $\alpha>0$.

i) Let $u_\varep\in W_0^{1, p}(\Omega; \mathbb{R}^m)$ be a weak solution of
$\mathcal{L}_\varep (u_\varep) =\text{\rm div} (h)$ in $\Omega$,
where $1<p<\infty$ and $h\in L^p(\Omega; \mathbb{R}^{m\times d})$.
Then
\begin{equation}\label{b-W-007}
\|\nabla u_\varep \|_{L^p(\Omega)}
\le C_p \, \| h\|_{L^p(\Omega)},
\end{equation}
where $C_p$ depends only on $p$, $\Omega$, and $A$.

ii) Let $u_\varep\in W^{1, p}(\Omega; \mathbb{R}^m)$ be a weak solution to
$\mathcal{L}_\varep (u_\varep) =\text{\rm div} (h)$ in $\Omega$ and $\frac{\partial u_\varep}{\partial \nu_\varep} 
=-n\cdot h$ on $\partial\Omega$,
where $1<p<\infty$ and $h\in L^p(\Omega; \mathbb{R}^{m\times d})$.
Then estimate (\ref{b-W-007}) holds with $C_p$ depending only on $p$, $\Omega$, and $A$.
\end{theorem}

\begin{proof}
Since the adjoint operator $\mathcal{L}_\varep^*$ satisfies the same conditions as
$\mathcal{L}_\varep$,
by a duality argument, we may assume that $p>2$.
By a real-variable argument (see \cite{Shen-2005,Geng-2012}), to prove (\ref{b-W-007}) for a fixed 
$p>2$, it suffices to establish two
weak reverse H\"older estimates for some $q>p$:

(i)\ \  if  $\mathcal{L}_\varep (u_\varep)=0$ in $2B$ and $2B\subset \Omega$, then 
\begin{equation}\label{reverse-H-1}
\left\{\average_{B} |\nabla u_\varep|^q \right\}^{1/q}
\le C \left\{\average_{2B} |\nabla u_\varep|^2 \right\}^{1/2},
\end{equation}
 
 (ii)\ \  if $\mathcal{L}_\varep (u_\varep)=0$ on $2B\cap \Omega$ with either 
 $u_\varep =0$ or $\frac{\partial u_\varep}{\partial \nu_\varep}=0$ on $2B\cap \partial\Omega$,
  where $B=B(x_0, r)$, $x_0\in \partial\Omega$ and
 $0<r<r_0=c_0\, \text{\rm diam}(\Omega)$, 
 then
\begin{equation}\label{reverse-H-2}
\left\{\average_{B\cap \Omega } |\nabla u_\varep|^q \right\}^{1/q}
\le C \left\{\average_{2B\cap\Omega} |\nabla u_\varep|^2 \right\}^{1/2}.
\end{equation}
Note that estimate (\ref{reverse-H-1}) is the interior $W^{1,p}$ estimate given by
 Lemma \ref{interior-W-1-p-lemma},
while (\ref{reverse-H-2}) follows from  the boundary $W^{1,p}$ estimates proved in Lemma \ref{b-W-lemma-1}.
\end{proof}

We are now in a position to give the proof of Theorems \ref{main-theorem-W-1-p} and \ref{main-theorem-W-1-p-N}.

\begin{proof}[\bf Proof of Theorems \ref{main-theorem-W-1-p} and \ref{main-theorem-W-1-p-N}]
For the Neumann condition the reduction of Theorem \ref{main-theorem-W-1-p-N} 
to Theorem \ref{b-W-1-p-theorem} may be found in \cite{Geng-2012, KLS1}.

For Dirichlet condition the reduction of Theorem \ref{main-theorem-W-1-p}
to Theorem \ref{b-W-1-p-theorem} is also more or less well known.
By considering $u_\varep -w$, where $w\in W^{1, p}(\Omega; \mathbb{R}^m)$
is the solution of $-\Delta w=0$ in $\Omega$ and $w=f$ on $\partial\Omega$,
it suffices to prove the theorem for the case $f=0$.
Here we have used the fact that the theorem holds for $\mathcal{L}_\varep =-\Delta$.
Next, in view of Theorem \ref{b-W-1-p-theorem}, we may further assume that $h=0$.
Finally, the case that $\mathcal{L}_\varep (u_\varep)=F$ in $\Omega$
and $u_\varep =0$ on $\partial\Omega$ may be handled by a duality argument.
Indeed, let $v_\varep$ be a solution of $\mathcal{L}_\varep^* (v_\varep)
=\text{\rm div} (h)$ in $\Omega$ and $v_\varep=0$ on $\partial\Omega$,
where $h=(h_i^\alpha) \in C_0^\infty (\Omega; \mathbb{R}^{m\times d})$.
Then
$$
\aligned
\left|\int_\Omega \frac{\partial u_\varep^\alpha}{\partial x_i} \cdot h_i^\alpha \right|
 & =\left| \int_\Omega F^\alpha \cdot v^\alpha\right|
\le \| F\|_{L^p(\Omega)} \|  v_\varep\|_{L^{p^\prime} (\Omega)}\\
&\le C\, \| F\|_{L^p(\Omega)} \|  \nabla v_\varep\|_{L^{p^\prime} (\Omega)}
\le 
C\, \| F\|_{L^p(\Omega)} \| h\|_{L^{p^\prime}(\Omega)},
\endaligned
$$
where we have used Poincar\'e inequality and $W^{1, p}$ estimates for $\mathcal{L}_\varep^*$.
By duality this gives $\|\nabla u_\varep\|_{L^p(\Omega)}
\le C\, \| F\|_{L^p(\Omega)}$.
\end{proof}



\section{Boundary Lipschitz estimates with Dirichlet condition
and Proof of Theorems \ref{main-theorem-Lip} and \ref{main-theorem-max}}
\setcounter{equation}{0}

In this section we establish the uniform boundary Lipschitz estimates for 
$\mathcal{L}_\varep$ in  bounded $C^{1, \alpha}$ domains
and give the proof of Theorem \ref{main-theorem-Lip}.
As in the case of interior Lipschitz estimates, our approach is based
on the general scheme outlined in Section 3.
However, modifications are needed to take into account the boundary contribution.
\begin{lemma}\label{b-L-lemma-1}
Suppose that $\mathcal{L}_0 (w)=0$ in $D_r$ and $w=f$ on $\Delta_r$ for some $0<r\le 1$.
Let 
$$
\aligned
G(t)= & \frac{1}{t} \inf_{\substack{M\in \mathbb{R}^{m\times d}\\ q\in \mathbb{R}^d}}
\bigg\{ \left(\average_{D_t} |w-Mx-q|^2\right)^{1/2}
+\| f-Mx-q\|_{L^\infty(\Delta_t)}\\
&\qquad  + t \|\nabla_{tan} \big (f-Mx-q\big)\|_{L^\infty(\Delta_t)}
+t^{1+\beta} \|\nabla_{tan} \big (f-Mx-q\big)\|_{C^{0, \beta}(\Delta_t)}
\bigg\}
\endaligned
$$
for $0<t\le r$,
where $\beta =\alpha/2$.
Then, there exists $\theta\in (0,1/4)$, depending only on $\mu$ and $(\alpha, K_0)$ in (\ref{phi}),
such that
\begin{equation}
 G(\theta r) \le (1/2)  G(r).
\end{equation}
\end{lemma}

\begin{proof} 
The lemma follows from  boundary $C^{1, \alpha}$ estimates for second-order elliptic systems
with constant coefficients. By rescaling we may assume $r=1$.
By choosing $q=w(0)$ and $M=\nabla w(0)$,
it is easy to see that for any $\theta\in (0,1/4)$,
$$
G(\theta) \le C\, \theta^{\beta } \| w\|_{C^{1,\beta}(D_\theta)}.
$$
By  boundary $C^{1, \alpha}$ estimates for $\mathcal{L}_0$, we obtain
$$
\| w\|_{C^{1,\beta}(D_\theta)}
\le C \left\{ \left(\average_{D_1} |w|^2\right)^{1/2}
+\| g\|_{L^\infty(\Delta_1)} +\|\nabla_{tan} g \|_{L^\infty(\Delta_1)}
+\| g\|_{C^{0, \beta} (\Delta_1)}\right\},
$$
where $C$ depends only on $\mu$ and $(\alpha, K_0)$.
It follows that
$$
G(\theta)
\le C\, \theta^{\beta}
\left\{ \left(\average_{D_1} |w|^2\right)^{1/2}
+\| g\|_{L^\infty(\Delta_1)} +\|\nabla_{tan} g \|_{L^\infty(\Delta_1)}
+\| g\|_{C^{0, \beta} (\Delta_1)}\right\}.
$$
Finally, since $\mathcal{L}_0 (Mx +q)=0$ for any $M\in \mathbb{R}^{m\times d}$ and
$q\in \mathbb{R}^m$, the estimate above implies that
$$
 G(\theta)\le C\, \theta^\beta G(1).
$$
The desired estimate follows by choosing $\theta\in (0,1/4)$ so small that $C\theta^\beta\le (1/2)$.
\end{proof}

\begin{lemma}\label{b-L-lemma-2}
Let $\mathcal{L}_\varep (u_\varep)=0$ in $D_{2r}$ and $u_\varep=f$ on $\Delta_{2r}$, where $0<\varep<r\le 1$.
Then there exists $w$ such that $\mathcal{L}_0 (w)=0$ in $D_{r}$,
$w=f$ on $\Delta_r$, and
\begin{equation}\label{b-L-2-0}
\aligned
\left\{ \average_{D_r} |u_\varep -w|^2 \right\}^{1/2}
 \le C \big[  \omega (\varep/r)\big]^{\frac23-\delta}
 &\bigg\{   \inf_{q\in \mathbb{R}^m}\bigg[
\bigg( \average_{D_{2r}}  |u_\varep-q |^2 \bigg)^{1/2} +\| f-q\|_{L^\infty(\Delta_{2r})} \bigg]\\
&  +r\, \| \nabla_{tan} f\|_{L^\infty(\Delta_{2r})}
+r^{1+\beta } \|\nabla_{tan} f\|_{C^{0, \beta }
(\Delta_{2r})} \bigg\},
\endaligned
\end{equation}
where $\delta\in (0,1/4)$, $\beta=\alpha/2$, and $\omega (t)=\omega_\sigma (t)$ is defined by (\ref{omega}).
The constant $C$ depends only on $\delta$, $\sigma$, $(\alpha, K_0)$ in (\ref{phi}), and $A$.
\end{lemma}

\begin{proof}
By rescaling we may assume  $r=1$.
For each $t\in [0,1/4)$, we construct a bounded $C^{1,\alpha}$ domain $\Omega_{1+t}$ in $\mathbb{R}^d$
such that (1) $D_1\subset\Omega_1 \subset \Omega_{1+t}\subset D_{3/2}$,
(2) there exists a $C^{1, \alpha}$ diffeomorphism $\Lambda_t: \partial\Omega_1 \to \partial\Omega_{1+t}$
with uniform bounds and
the property that $|\Lambda_t (x)-x|\le C\, t$ for any $x\in \partial\Omega_1$,
and (3) for each $x\in \partial\Omega_1$, $B(x, ct)\cap D_2 \subset \Omega_{1+t}$.

Let $w=w_t$ be the solution of Dirichlet problem:
$\mathcal{L}_0 (w)=0$ in $\Omega_{1+t}$ and $w=u_\varep$ on $\partial\Omega_{1+t}$.
Note that $\mathcal{L}_0 (w) =0$ in $D_1$ and $w=f$ on $\Delta_1$.
We will show that $w$ satisfies the estimate (\ref{b-L-2-0}) for some suitable choice of $t$.

Let $v_\varep$ be the solution of $\mathcal{L}_\varep (v_\varep)=0$ in $\Omega_1$
and $v_\varep=w$ on $\partial \Omega_1$.
Since $\mathcal{L}_\varep (u_\varep -v_\varep)=0$ in $\Omega_1$,
by the H\"older estimate (\ref{H-estimate}) for $\mathcal{L}_\varep$,
\begin{equation}\label{b-L-2-1}
\aligned
\| u_\varep -v_\varep\|_{L^2(D_1)}
& \le C\, \| u_\varep -w\|_{C^\kappa (\partial \Omega_1)}
\le C\, t^{\gamma-\kappa} \, \| u_\varep\|_{C^\gamma (D_{3/2})}\\
&\le C\, t^{\gamma-\kappa}
\left\{ \left(\average_{D_{2}} |u_\varep|^2 \right)^{1/2}
+\| f\|_{L^\infty(\Delta_2)}
+\| \nabla_{tan} f\|_{L^\infty(\Delta_{2})} \right\},
\endaligned
\end{equation}
where $0<\kappa<\gamma<1$.
The fact that $\Lambda_t (x)-x|\le C\, t$ for $t\in \partial\Omega_1$ is used
for the second inequality in (\ref{b-L-2-1}).
Next, by Theorem \ref{rate-theorem-1}, we see that
\begin{equation}\label{b-L-2-2}
\aligned
\| v_\varep -w\|_{L^2(D_1)}
& \le C\,  \big[\omega(\varep)\big]^{2/3} \, \| w\|_{C^{1, \kappa}({\partial\Omega_1})}\\
& \le C\, \big[\omega(\varep)\big]^{2/3}
t^{\gamma-1-\kappa} \left\{ \| w\|_{C^{\gamma} (\Omega_{1+t})} 
+\|  g\|_{C^{ 1,\kappa}(\Delta_2)} \right\}\\
&\le C\, \big[\omega(\varep)\big]^{2/3}
t^{\gamma-1-\kappa} \left\{ \| u_\varep \|_{C^{\gamma} (\Omega_{1+t})} +\|  g\|_{C^{1,\kappa}(\Delta_2)} \right\}\\
&\le C\, \big[\omega(\varep)\big]^{2/3}
t^{\gamma-1-\kappa} \left\{ \left(\average_{D_2} |u_\varep|^2\right)^{1/2}
+ \| g\|_{C^{1, \kappa}(\Delta_2)} \right\},
\endaligned
\end{equation}
where we have used  the boundary $C^{1,\kappa}$ estimates for $\mathcal{L}_0$ 
for the second inequality and H\"older estimates for the third.
We point out that for the second inequality in (\ref{b-L-2-2}) we also have used the fact
$B(x, ct)\cap D_2\subset \Omega_{1+t}$ for any $x\in \partial\Omega_1$.
It follows from (\ref{b-L-2-1}) and (\ref{b-L-2-2}) that
$$
\aligned
\| u_\varep -w\|_{L^2(D_1)}
& \le \| u_\varep -v_\varep \|_{L^2(D_1)} +\| v_\varep -w\|_{L^2(D_1)}\\
&\le C t^{\gamma-\kappa}
\left\{ 1+ t^{-1} \big[ \omega (\varep) \big]^{2/3}\right\}
\left\{ \left(\average_{D_2} |u_\varep|^2\right)^{1/2}
+\| g\|_{C^{1, \kappa}(\Delta_2)} \right\},\\
&\le C \, \big[\omega (\varep)\big]^{\frac{2}{3}-\delta} 
\left\{ \left(\average_{D_2} |u_\varep|^2\right)^{1/2}
+\| g\|_{C^{1, \kappa}(\Delta_2)} \right\},
\endaligned
$$
where we have chosen $t=c\big[\omega (\varep)\big]^{2/3}\in (0,1/4)$, $\kappa=(3/4)\delta$ and
$\gamma=1-(3/4)\delta$.
This yields the estimate (\ref{b-L-2-0}), as $\mathcal{L}_\varep (u_\varep -q)=0$ in $D_2$ for any
$q\in \mathbb{R}^m$.
\end{proof}

We are now ready to prove the boundary Lipschitz estimates for $\mathcal{L}_\varep$.

\begin{theorem}\label{b-L-theoem-3}
Suppose that $A(y)$ satisfies the same conditions as in Theorem \ref{main-theorem-Lip}.
Let $u_\varep$ be a weak solution of
$\mathcal{L}_\varep (u_\varep)=0$ in $D_{2r}$ and $u_\varep =f$ on $\Delta_{2r}$
for some $0<r\le 1$.
Then
\begin{equation}
\aligned
\|\nabla u_\varep\|_{L^\infty(D_r)}
&\le C\bigg\{ \frac{1}{r} \left(\average_{D_{2r}} |u_\varep|^2 \right)^{1/2}
+r \| f\|_{L^\infty(\Delta_{2r})}\\
& \qquad \qquad \qquad+\|\nabla_{\tan} f\|_{L^\infty(\Delta_{2r})}
+r^\beta \|\nabla_{tan} f\|_{C^{0,\beta} (\Delta_{2r})} \bigg\},
\endaligned
\end{equation}
where $\beta=\alpha/2$ and $C$ depends only on $(\alpha,K_0)$ and $A$.
\end{theorem}

\begin{proof}
By rescaling we may assume that $r=1$.
Let
$$
\aligned
H(t)= & t^{-1} \inf_{\substack{M\in \mathbb{R}^{m\times d}\\ q\in \mathbb{R}^d}}
\bigg\{ \left(\average_{D_t} |u_\varep-Mx-q|^2\right)^{1/2}
+\| f-Mx-q\|_{L^\infty(\Delta_t)}\\
&\qquad  + t \|\nabla_{tan} \big (f-Mx-q\big)\|_{L^\infty(\Delta_t)}
+t^{1+\beta} \|\nabla_{tan} \big (f-Mx-q\big)\|_{C^{0, \beta}(\Delta_t)}
\bigg\},
\endaligned
$$
where $0<t\le 1$.
For each $\varep<t<1$, let $w=w_t$ be a solution of $\mathcal{L}_0 (w)=0$
in $D_t$ with $w=f$ on $\Delta_t$, given by Lemma \ref{b-L-lemma-2}.
For $0<s\le t$, let $G(s)$ be defined as $H(t)$, but with $u_\varep$ replaced by $w$ and $t$ replaced by $s$.
Observe that
$$
\aligned
H(\theta t) 
&\le G(\theta t) + \frac{1}{\theta t} \left\{ \average_{D_{\theta t}} |u_\varep -w|^2 \right\}^{1/2}\\
& \le \frac12 G(t) +\frac{1}{\theta t} \left\{ \average_{D_{\theta t}} |u_\varep -w|^2 \right\}^{1/2}\\
&\le \frac12 H(t)
+ \frac{C}{t} \left\{ \average_{D_{t}} |u_\varep -w|^2 \right\}^{1/2},
\endaligned
$$
where we have used Lemma \ref{b-L-lemma-1} for the second inequality.
This, together with Lemma \ref{b-L-lemma-2}, gives
\begin{equation}\label{b-L-3-10}
\aligned
H(\theta t) \le \frac12 H(t)
+ C \big[  \omega (\varep/t)\big]^{\frac23-\delta}
 &\bigg\{   t^{-1} \inf_{q\in \mathbb{R}^m}\bigg[
\bigg( \average_{D_{2t}}  |u_\varep-q |^2 \bigg)^{1/2} +\| f-q\|_{L^\infty(\Delta_{2t})} \bigg]\\
&  + \| \nabla_{tan} f\|_{L^\infty(\Delta_{2t})}
+t^{\beta } \|\nabla_{tan} f\|_{C^{0, \beta }
(\Delta_{2t})} \bigg\}
\endaligned
\end{equation}
for any $\varep<t\le 1$.
Let  $r_j=\theta^{j+1}$ for $0\le j\le \ell$, where $\ell$ is chosen so that
$\theta^{\ell+2}<\varep\le \theta^{\ell +1}$.
Let
$$
F_j =H(r_j) \quad \text{ and } \quad p_j=|M_j|,
$$
 where $M_j \in \mathbb{R}^{m\times d}$ is a matrix such that
$$
\aligned
H(r_j)
= & r_j^{-1} \inf_{ q\in \mathbb{R}^d}
\bigg\{ \left(\average_{D_{r_j}} |u_\varep-M_jx-q|^2\right)^{1/2}
+\| f-M_j x-q\|_{L^\infty(\Delta_{r_j})}\\
&\qquad  + r_j \|\nabla_{tan} \big (f-M_jx-q\big)\|_{L^\infty(\Delta_{r_j})}
+r_j^{1+\beta} \|\nabla_{tan} \big (f-M_jx-q\big)\|_{C^{0, \beta}(\Delta_{r_j})}
\bigg\}.
\endaligned
$$
In view of (\ref{b-L-3-10}) we obtain 
\begin{equation}\label{b-L-3-20}
F_{j+1} \le \frac12 F_j + C \big[ \omega (\varep \theta^{-j-1})\big]^{\frac23-\delta} \big\{ F_{j-1} +p_{j-1} \big\}.
\end{equation}
Also observe that since $D_{r} $ satisfies the interior cone condition,
$$
\aligned
& |M_{j+1} -M_j|
\le \frac{C}{r_{j+1}} \inf_{q\in \mathbb{R}^m}
\left\{ \average_{D_{r_{j+1}}} |(M_{j+1} -M_j) x-q |^2 \right\}^{1/2}\\
&\le \frac{C}{r_{j+1}}
\inf_{q\in \mathbb{R}^m}
\left\{ \average_{D_{r_{j+1}}} |u_\varep - M_{j+1}x-q|^2 \right\}^{1/2}
+
\frac{C}{r_{j+1}}\inf_{q\in \mathbb{R}^m}
\left\{ \average_{D_{r_{j+1}}} |u_\varep - M_{j}x-q|^2 \right\}^{1/2}\\
& \le  C \, ( F_j +F_{j+1} ).
\endaligned
$$
It follows that
\begin{equation}
p_{j+1} =|M_{j+1}|\le |M_j| + C\, (F_j + F_{j+1})
 = p_j +C\, (F_j + F_{j+1}).
\end{equation}
Recall that  the condition (\ref{decay-condition}) implies that
$$
\sum_{j=0}^\ell \big[ \omega (\varep \theta^{-j-1})\big]^{\frac23 -\delta} \le C
\int_0^1 \big[ \omega (t) \big]^{\frac23 -\delta} \, \frac{dt}{t} <\infty,
$$
for some $\sigma, \delta \in (0,1)$.
This allows us to apply Lemma \ref{main-lemma-1}  to obtain
$$
\aligned
F_j  +p_j & \le C\, \big\{ p_0 +F_0 +F_1 \big\}\\
&\le C \left\{ \left(\average_{D_1} |u_\varep|^2 \right)^{1/2}
+\| f\|_{C^{1, \beta} (\Delta_1)} \right\}
\endaligned
$$
for any $0\le j\le \ell$.
As a result, we see that for any $\varep< t<1/4$,
\begin{equation}\label{b-L-3-30}
\aligned
\left(\average_{D_t} |\nabla u_\varep|^2\right)^{1/2}
&\le 
\frac{C}{t} \inf_{q\in \mathbb{R}^m}
\left\{ \left(\average_{D_{2t}} |u_\varep -q|^2\right)^{1/2}
+\| f-q \|_{L^\infty(\Delta_{2t})} \right\} + C \| \nabla_{tan} f \|_{L^\infty (\Delta_{2t})}\\
&\le C \left\{ \left(\average_{D_1} |u_\varep|^2\right)^{1/2} 
+\| f\|_{C^{1,\beta}(\Delta_1)} \right\},
\endaligned
\end{equation}
where we have used  Caccipoli's inequality for the first inequality.

Finally, since $A(y)$ is H\"older continuous, we may use apply the classical
boundary Lipschitz estimates for $\mathcal{L}_1$ and a blow-up argument to obtain 
$$
\aligned
\|\nabla u_\varep\|_{L^\infty(D_\varep)}
& \le \frac{C}{\varep} \inf_{q\in \mathbb{R}^m}
\left\{ \left(\average_{D_{2\varep}} |u_\varep -q|^2\right)^{1/2} +\| f-q \|_{L^\infty(\Delta_{2\varep})} \right\}
+C\ \| \nabla_{tan} f \|_{C^\sigma (\Delta_{2\varep})}\\
&\le C \left\{ \left(\average_{D_1} |u_\varep|^2\right)^{1/2} 
+\| f\|_{C^{1,\beta}(\Delta_1)} \right\},
\endaligned
$$
where we have used (\ref{b-L-3-30}) with $t=2\varep$ for the second inequality.
Consequently, we see that 
\begin{equation}\label{b-L-3-40}
\left(\average_{D_t} |\nabla u_\varep|^2\right)^{1/2}
\le C \left\{ \left(\average_{D_1} |u_\varep|^2\right)^{1/2} 
+\| f\|_{C^{1,\beta}(\Delta_1)} \right\}
\end{equation}
holds for any $0<t<1/4$.
This, together with the interior Lipschitz estimates proved in Section 4, yields
\begin{equation}\label{b-L-3-50}
\| \nabla u_\varep \|_{L^\infty(D_1)}
\le C \left\{ \left(\average_{D_2} |u_\varep|^2\right)^{1/2} 
+\| f\|_{C^{1,\beta}(\Delta_2)} \right\}.
\end{equation}
The proof is complete.
\end{proof}

We now give the proof of Theorem \ref{main-theorem-Lip}.

\begin{proof}[\bf Proof of Theorem \ref{main-theorem-Lip}]
It suffices to show that if $\mathcal{L}_\varep (u_\varep)=F$ in $D_{2r}$
and $u_\varep=f $ on $\Delta_{2r}$ for some $0<r<1$, then
\begin{equation}\label{b-L-local}
\aligned
\| \nabla u_\varep\|_{L^\infty(D_r)}
\le C r^{-1} \| u_\varep\|_{L^\infty(D_{2r})}
& +C \|\nabla_{tan} f\|_{L^\infty(\Delta_{2r})}  +C r^\beta \|\nabla_{\tan} f\|_{C^{0, \beta}(\Delta_{2r})}\\
&+ Cr^\beta \sup_{\substack{x\in D_{2r}\\ 0<t<r}}
t^{1-\beta} \average_{B(x,t)\cap D_{2r}}
|F|.
\endaligned
\end{equation}
Estimate (\ref{Lip-estimate-0}) follows from (\ref{b-L-local}) and the interior Lipschitz estimate 
 by a simple covering argument.

To prove (\ref{b-L-local}), we may assume that $r=1$ and $d\ge 3$.
The case $F=0$ is already proved in the last lemma.
The general case may be handled by the use of Green functions.
Indeed, let $\Omega$ be a bounded $C^{1, \alpha}$ domain in $\mathbb{R}^d$
such that $D_{3/2}\subset \Omega\subset D_2$.
Let $G_\varep (x,y)$ denote the matrix of Green functions for $\mathcal{L}_\varep $ in 
$\Omega$, with pole at $y$.
By the boundary H\"older estimates in \cite{Shen-2014},
we know $|G_\varep (x,y)|\le C\, |x-y|^{2-d}$ for any $x,y\in \Omega$.
Since $\mathcal{L}_\varep \big( G_\varep (\cdot, y)\big)=0$ 
in $\Omega \setminus \{ y\}$ and $G(x, y)=0$ for $x\in \partial\Omega$,
we may use the boundary Lipschitz estimate in the last lemma to show that
$|\nabla_x G_\varep (x,y)|\le C\, |x-y|^{1-d}$ for any $x,y\in \Omega$.
One then considers $u_\varep -v_\varep$ in $D_2$, where
$v_\varep (x)=\int_\Omega G_\varep (x,y) F(y)\, dy$.
The rest of the argument is similar to that in the proof of Theorem \ref{i-L-theorem}.
We omit the details.
\end{proof}

\begin{proof}[\bf Proof of Theorem \ref{main-theorem-max}]

The non-tangential maximal function of $u_\varep$ is defined by
\begin{equation}\label{non-tan}
(u_\varep)^* (Q)=\sup \big\{ |u_\varep (x)|: \
x\in \Omega \text{ and } |x-Q|<C_0\,  \text{\rm dist} (x, \partial\Omega) \big\}
\end{equation}
for $Q\in \partial\Omega$,
where $C_0=C_0(\Omega)$ is sufficiently large.
Suppose that $\mathcal{L}_\varep (u_\varep) =0$ in $\Omega$ and $u_\varep =f$ on $\partial\Omega$.
It is well known that the estimate (\ref{P-estimate}) implies that
$\| u_\varep\|_{L^\infty (\Omega)} \le C\, \| f\|_{L^\infty(\partial\Omega)}$, and
$$
(u_\varep)^* (Q) \le C\, \mathcal{M}_{\partial\Omega} (f) (Q) \qquad \text{ for any } Q\in \partial\Omega,
$$
where $\mathcal{M}_{\partial\Omega} (f)$ denotes the Hardy-Littlewood maximal 
function of $f$ on $\partial\Omega$.
It follows that
$$
\| (u_\varep)^*\|_{L^p(\partial\Omega)} \le C_p \, \| f\|_{L^p(\partial\Omega)}
$$
for any $1<p<\infty$.
\end{proof}



\section{Boundary Lipschitz estimates with Neumann conditions and proof
of Theorem \ref{main-theorem-Lip-N}}
\setcounter{equation}{0}

In this section we establish the uniform Lipschitz estimates with Neumann boundary conditions
and give the proof of Theorem \ref{main-theorem-Lip-N}.
Throughout the section we will assume that $A$ satisfies the same conditions as 
in Theorem \ref{main-theorem-Lip-N}.

Let $D_r$ and $\Delta_r$ be defined as in  (\ref{Delta}) and $(\alpha, K_0)$ given in (\ref{phi}).

\begin{lemma}\label{Lip-N-lemma-1}
Suppose that $\mathcal{L}_0 (w)=0$ in $D_{2r}$ and $\frac{\partial w}{\partial \nu_0}=g$
on $\Delta_{2r}$. 
Let
$$
\aligned
\Psi(t)
= &\frac{1}{t}\inf_{\substack{M\in \mathbb{R}^{m\times d}\\ q\in \mathbb{R}^d}}
\bigg\{ \left(\average_{D_t} |w-Mx -q|^2\right)^{1/2}
+t \left\| \frac{\partial}{\partial\nu_0} \big( w-Mx\big) \right\|_{L^\infty(\Delta_t)}\\
&\qquad\qquad\qquad\qquad \qquad\qquad
+t^{1+\beta} \left\| \frac{\partial}{\partial \nu_0} \big (w-Mx\big) \right\|_{C^{0,\beta} (\Delta_t)}
\bigg\},
\endaligned
$$
for $0<t\le r$, where $\beta=\alpha/2$.
Then there exists $\theta\in (0,1/4)$, depending only on $\mu$, $\alpha$ and
$K_0$, such that
\begin{equation}\label{NP-Lip-8-1}
 \Psi(\theta r) \le (1/2) \Psi (r).
\end{equation}
\end{lemma}

\begin{proof}
The lemma follows from boundary $C^{1,\alpha}$ estimates with Neumann conditions 
for second-order elliptic systems with constant coefficients.
The argument is similar to that in the case of Dirichlet condition.
We leave the details to the reader.
\end{proof}

\begin{lemma}\label{decay-rate-lemma-8}
Let $\sigma\in (0,1)$ and
\begin{equation}\label{NP-eta}
\eta (t) =\eta_\sigma (t)= \left\{ \Theta_\sigma (t^{-1}) + \sup_{T\ge t^{-1} }\langle |\psi-\nabla \chi_T|\rangle \right\}^{1/2}.
\end{equation}
Suppose that there exist $C_0>0$ and $N>3$ such that $\rho (R)\le C_0 \big[\log R\big]^{-N}$ for all $R\ge 2$.
Then $\int_0^1 \eta (t)\, \frac{dt}{t} <\infty$.
\end{lemma}

\begin{proof}
Recall that if there exist $C_0>0$ and $N>1$ such that
$\rho(R)\le C_0 \big[ \log R \big]^{-N}$ for all $R\ge 2$, then 
$$
\Theta_\sigma (T)\le C_\sigma \big[\log T\big]^{-N}\quad
\text{ and } \quad
\langle |\psi-\nabla \chi_T| \rangle
\le C\, \big[\log T\big]^{-N+1}
$$
 for all $T\ge 2$ (see the proof of Lemma \ref{Dini-lemma}). This gives
 $$
 \eta (t) \le C \, \big[\log (1/t) \big]^{(1-N)/2} \qquad \text{ for } t\in (0,1/2),
 $$
 from which the lemma follows readily.
 \end{proof}
 
\begin{lemma}\label{Lip-N-lemma-2}
Suppose that $A$ satisfies the same conditions as in Theorem \ref{main-theorem-Lip}.
Let $u_\varep$ be a weak solution of $\mathcal{L}_\varep (u_\varep)=0$ in $D_{2r}$
with $\frac{\partial u_\varep}{\partial\nu_\varep} =g$ on $\Delta_{2r}$,
where $0<\varep<r\le 1$ and $g\in C^{1,\beta} (\Delta_{2r})$.
Then there exists $w\in H^1(D_r;\mathbb{R}^m)$ such that
$\mathcal{L}_0 (w)=0$ in $D_r$, $\frac{\partial w}{\partial \nu_0} =g$
on $\Delta_{2r}$, and
\begin{equation}\label{NP-Lip-8-2}
\left\{ \average_{D_r} |u_\varep -w|^2 \right\}^{1/2}
\le C \, \eta \left( \frac{\varep}{r}\right)\bigg\{
\inf_{q\in \mathbb{R}^m}
 \left(\average_{D_{2r}} |u_\varep -q|^2\right)^{1/2}
+r\, \| g\|_{L^\infty(\Delta_{2r})}\bigg\}
\end{equation}
where $\beta=\alpha/2$ and $\eta(t)$ is given by (\ref{NP-eta}).
The constant $C$ depends only on $\alpha$, $K_0$, and $A$.
\end{lemma}

\begin{proof}
By rescaling we may assume $r=1$.
By subtracting a constant we may assume that $\int_{D_2} u_\varep =0$.
Using Cacciopoli's inequality
\begin{equation}\label{NP-C}
\int_{D_{3/2}} |\nabla u_\varep|^2
\le C \left\{ \int_{D_2} |u_\varep|^2 +\int_{\Delta_2} |g|^2 \right\}
\end{equation}
and the co-area formula,
it is not hard to see that there exists a $C^{1, \alpha}$ domain $\Omega$
such that $D_1\subset \Omega\subset D_{3/2}$ and
\begin{equation}\label{NP-Lip-8-4}
\int_{\partial\Omega} |\nabla u_\varep|^2
\le  C \left\{ \int_{D_2} |u_\varep|^2 +\int_{\Delta_2} |g|^2 \right\}.
\end{equation}
Now let $w$ be the weak solution of $\mathcal{L}_\varep (w)=0$ in $\Omega$
with $\frac{\partial w}{\partial \nu_0} =\frac{\partial u_\varep}{\partial\nu_\varep}$
on $\partial\Omega$ and $\int_\Omega w=\int_\Omega u_\varep$.
It follows from Theorem \ref{NP-rate-theorem-2}
\begin{equation}\label{NP-Lip-8-5}
\| u_\varep -w\|_{L^2(\Omega)}
\le C\, \eta (\varep) \| \nabla u_\varep\|_{L^2(\partial\Omega)}
\le C\, \eta (\varep)  \left\{ \| u_\varep\|_{L^2(D_2)} +\| g\|_{L^2(\Delta_2)} \right\},
\end{equation}
where we have used (\ref{NP-Lip-8-4}) for the last inequality.
Since $\int_{D_2} u_\varep =0$, this yields (\ref{NP-Lip-8-2}).
\end{proof}

\begin{theorem}\label{NP-Lip-theorem-8}
Suppose that $A$ satisfies the same conditions as in Theorem \ref{main-theorem-Lip-N}.
Let $u_\varep\in H^1(D_{2r}; \mathbb{R}^m)$ be a weak solution
of $\mathcal{L}_\varep (u_\varep)=0$ in $D_{2r}$ with $\frac{\partial u_\varep}{\partial\nu_\varep}
=g$ on $\Delta_{2r}$.
Then
\begin{equation}\label{NP-Lip-8-10}
\|\nabla u_\varep\|_{L^\infty(D_r)}
\le C\left\{ \left(\average_{D_{2r}} |\nabla u_\varep|^2 \right)^{1/2}
+\| g\|_{L^\infty(\Delta_{2r})}
+r^\beta \| g\|_{C^{0, \beta} (\Delta_{2r})} \right\},
\end{equation}
where $\beta=\alpha/2$ and $C$ depends only on $(\alpha, K_0)$ and $A$.
\end{theorem}

\begin{proof}
With Lemmas \ref{Lip-N-lemma-1},  \ref{decay-rate-lemma-8} and \ref{Lip-N-lemma-2} at our disposal,
the theorem follows by the same line of argument as in the case of Dirichlet condition.
By rescaling we may assume $r=1$.
Let
$$
\aligned
\Phi (t)= &
t^{-1}\inf_{\substack{ M\in \mathbb{R}^{m\times d}\\ q\in \mathbb{R}^m}}
\bigg\{ \left(\average_{D_t} |u_\varep -Mx -q|^2 \right)^{1/2}
+t \left\| g -\frac{\partial}{\partial \nu_0} \big( Mx \big)\right\|_{L^\infty(\Delta_t)} \\
&\qquad\qquad\qquad\qquad\qquad
+t^{1+\beta} \left\| g -\frac{\partial}{\partial\nu_0} \big( Mx \big) \right\|_{C^{0,\beta}
(\Delta_t)} \bigg\}
\endaligned
$$
for $0<t\le 1$.
For each $\varep<t\le 1$, let
$w=w_t$ be the solution of $\mathcal{L}_0(w)=0$ in $D_t$ with $\frac{\partial w}{\partial \nu_0}
=g$ on $\Delta_t$, given by Lemma \ref{Lip-N-lemma-2}.
As in the case of Dirichlet condition, it follows from Lemma \ref{Lip-N-lemma-1} that
$$
\Phi (\theta t) \le \frac12 \Phi (t) +\frac{C}{t} \left\{ \average_{D_t} |u_\varep -w|^2\right\}^{1/2},
$$
where $\theta \in (0,1/4)$ is given by Lemma \ref{Lip-N-lemma-1}.
In view of Lemma \ref{Lip-N-lemma-2}, this leads to
\begin{equation}\label{NP-Lip-8-20}
\Phi (\theta t) \le \frac12 \Phi (t)
+C \, \eta (\varep/t) \left\{ \frac{1}{t} \inf_{q\in \mathbb{R}^m}
\left(\average_{D_{2t}} |u_\varep -q|^2 \right)^{1/2} 
+\| g\|_{L^\infty(\Delta_{2t}) }\right\}.
\end{equation}

Now, let $r_j=\theta^{j+1}$ for $0\le j\le \ell$, where
$\ell$ is chosen so that $\theta^{\ell+1}<\varep\le \theta^{\ell +1}$.
Let
$$
F_j =\Phi (r_j) \quad \text{ and } \quad p_j =|M_j|,
$$
where $M_j\in \mathbb{R}^{m\times d}$ is a matrix such that
$$
\aligned
\Phi (r_j)= &
r^{-1}_j 
\bigg\{ \inf_{q\in \mathbb{R}^m}
\left(\average_{D_{r_j}} |u_\varep -M_j x -q|^2 \right)^{1/2}
+r_j \left\| g -\frac{\partial}{\partial\nu_0} \big( M_j x \big)\right\|_{L^\infty(\Delta_{r_j})} \\
&\qquad\qquad\qquad\qquad\qquad
+r_j^{1+\beta} \left\| g-\frac{\partial}{\partial \nu_0} \big( M_j x \big) \right\|_{C^{0,\beta}
(\Delta_{r_j})} \bigg\}.
\endaligned
$$
It follows from the estimate (\ref{NP-Lip-8-20}) that
\begin{equation}\label{NP-Lip-8-30}
F_{j+1} \le \frac12 F_j + C\, \eta (\varep 2^{-j-1}) \big\{ F_{j-1} + p_{j-1} \big\}.
\end{equation}
As in the proof of Theorem \ref{b-L-theoem-3}, we also have
\begin{equation}
p_{j+1} \le p_j +C \big\{ F_j +F_{j+1} \big\}.
\end{equation}
Furthermore, by Lemma \ref{decay-rate-lemma-8},
$$
\sum_{j=1}^\ell \eta (\varep \theta^{-j-1}) \le C \int_0^1 \eta (t)\, \frac{dt}{t} <\infty.
$$
Consequently, we may apply Lemma \ref{main-lemma-1} to obtain 
$$
\aligned
F_j + p_j  &\le C \big\{ p_0 +F_0 +F_1\big\}\\
&\le C \left\{ \left(\average_{D_1} |u_\varep|^2\right)^{1/2} + \| g\|_{C^\beta (\Delta_1)} \right\}.
\endaligned
$$
This, together with the Cacciopoli's inequality, yields that for any $\varep<t<(1/4)$,
\begin{equation}\label{NP-Lip-8-40}
\left\{ \average_{D_t} |\nabla u_\varep|^2 \right\}^{1/2}
\le C \left\{ \left(\average_{D_1} |u_\varep|^2\right)^{1/2} + \| g\|_{C^\beta (\Delta_1)} \right\}.
\end{equation}
As in the case of Dirichlet condition,
we may use a blow-up argument and (\ref{NP-Lip-8-40}) to show that the estimate above in fact
holds for any $0<t<(1/4)$.
Finally, we observe that the estimate (\ref{NP-Lip-8-10}) follows from (\ref{NP-Lip-8-40})
and the interior Lipschitz estimates.
\end{proof}

\begin{remark}\label{remark-NP}
{\rm 
Let $\Omega$ be a bounded $C^{1,\alpha}$ domain in $\mathbb{R}^d$.
Let $N_\varep(x,y)$ denote the matrix of Neumann functions for $\mathcal{L}_\varep$
in $\Omega$, with pole at $y$; i.e., 
$$
\left\{
\aligned
\mathcal{L}_\varep \big\{ N_\varep (\cdot, y) \big\} & = I_{m\times m} \delta_ y(x) &\quad &\text{ in } \Omega,\\
\frac{\partial}{\partial \nu_\varep} \big\{ N_\varep (\cdot, y) \big\} 
& =-|\partial\Omega|^{-1} I_{m\times m} &\quad & \text{ on } \partial\Omega,
\endaligned
\right.
$$
where $I_{m\times m}$ denotes the $m\times m$ identity matrix.
Suppose that $A$ satisfies the conditions in Theorem \ref{main-theorem-Lip-N}.
Since $A^*$ also satisfies the same conditions, it follows from Theorem \ref{NP-Lip-theorem-8}
that if $d\ge 3$,
\begin{equation}\label{N-F-estimate}
\left\{
\aligned
|N_\varep (x,y)| & \le C\, |x-y|^{2-d},\\
|\nabla_x N_\varep (x,y)| +|\nabla_y N_\varep (x,y)| & \le C\, |x-y|^{1-d},\\
|\nabla_x\nabla_y N_\varep (x,y)| & \le C\, |x-y|^{-d}
\endaligned
\right.
\end{equation}
for any $x,y\in \Omega$, $x\neq y$, where $C$ depends only on $A$ and $\Omega$.
We refer the reader to \cite{KLS1} for the proof in the periodic setting.
}
\end{remark}
We now give the proof of Theorem \ref{main-theorem-Lip-N}

\begin{proof}[\bf Proof of Theorem \ref{main-theorem-Lip-N}]
It suffices to show that if $\mathcal{L}(u_\varep) =F$ in $D_{2r}$ and
$\frac{\partial u_\varep}{\partial \nu_\varep} =g$ on $\partial\Omega$ for some
$0<r<1$, then
\begin{equation}\label{NP-Lip-8-100}
\aligned
\|\nabla u_\varep\|_{L^\infty(D_r)}
 &\le C\left(\average_{D_{2r}} |\nabla u_\varep|^2\right)^{1/2}
 + C \, \| g\|_{L^\infty(\Delta_{2r})} + C\, r^\beta \|g\|_{C^{0, \beta}(\Delta_{2r})}\\
 &\qquad\qquad\qquad\qquad
 +C r^\beta \sup_{\substack{x\in D_{2r}\\ 0<t<r }} t^{1-\beta}
 \average_{B(x,t)\cap D_{2r}} |F|.
 \endaligned
 \end{equation}
By rescaling we may assume $r=1$.
The case $F=0$ is given by Theorem \ref{NP-Lip-theorem-8}.
To deal with the general case, we assume $d\ge 3$ (the case $d=2$ is reduced to the case $d=3$ by
adding a dummy variable).
Let $\Omega$ be a bounded $C^{1,\alpha}$ domain such that $D_{3/2}\subset \Omega\subset D_2$.
Let $N_\varep (x,y)$ denote the matrix of Neumann functions for $\mathcal{L}$ in 
$\Omega$, with pole at $y$.
Let $v_\varep (x)=\int_\Omega N_\varep (x,y) F(y)\, dy$.
Note that by (\ref{N-F-estimate}),
$$
|\nabla v_\varep (x)|\le C\int_\Omega \frac{|F(y)|}{|x-y|^{d-1}}\, dy
\le C  \sup_{\substack{x\in D_{2}\\ 0<t<1 }} t^{1-\sigma}
 \average_{B(x,t)\cap D_{2}} |F|.
 $$
By considering $u_\varep -v_\varep$, 
we may reduce the general case to the case $F=0$.
\end{proof}



\bibliography{as.bib}

\medskip

\begin{flushleft}

Scott N. Armstrong

Ceremade (UMR CNRS 7534)

Universit\'e Paris-Dauphine

Paris, France

\end{flushleft}

\begin{flushleft}
Zhongwei Shen

 Department of Mathematics
 
University of Kentucky

Lexington, Kentucky 40506,
USA. 


E-mail: zshen2@uky.edu
\end{flushleft}

\medskip

\noindent \today

\end{document}